\documentclass[11pt,reqno]{amsart}

\usepackage{amssymb}
\usepackage{color}

\usepackage{latexsym}
\usepackage{graphicx}

\newcommand{\ispa}[1]{\langle \,#1 \,\rangle }

\newcommand{\re}{\mathop{{\rm Re}}\nolimits}

\newcommand{\im}{\mathop{{\rm Im}}\nolimits}
  
\newcommand{\ol}{\overline}  
\newcommand{\mb}{\mathbb}

\newcommand{\ve}{\varepsilon}

\newcommand{\ai}{{\rm Ai}}
\newtheorem{thm}{{\sc Theorem}}[section]
\newtheorem{cor}[thm]{{\sc Corollary}}
\newtheorem{lem}[thm]{{\sc Lemma}}
\newtheorem{prop}[thm]{{\sc Proposition}}

\newenvironment{rem}{\medskip\noindent{\it Remark$:$\/} }{\medskip}

\begin{document}

\title
{Asymptotic behavior of quantum walks on the line}

\author{Toshikazu Sunada}
\address{
Department of Mathematics, Meiji University, 
Higashimita 1-1-1, Tama-ku, Kawasaki, 214-8571 Japan
}
\email{sunada@isc.meiji.ac.jp}
\author{Tatsuya Tate}
\address{Graduate School of Mathematics, Nagoya University, 
Furo-cho, Chikusa-ku, Nagoya, 464-8602 Japan}
\email{tate@math.nagoya-u.ac.jp}
\thanks{The first author is partially supported by JSPS Grant-in-Aid for Scientific Research (No. 21340039).}
\thanks{The second author is partially supported by JSPS Grant-in-Aid for Scientific Research (No. 21740117).}
\date{\today}

\renewcommand{\thefootnote}{\fnsymbol{footnote}}
\renewcommand{\theequation}{\thesection.\arabic{equation}}
\renewcommand{\labelenumi}{{\rm (\arabic{enumi})}}
\renewcommand{\labelenumii}{{\rm (\alph{enumii})}}

\begin{abstract}
This paper gives various asymptotic formulae for the transition probability associated with discrete time quantum walks on the real line.
The formulae depend heavily on the `normalized' position of the walk.
When the position is in the support of the weak-limit distribution obtained by Konno (\cite{Ko1}), one observes, in addition to the limit distribution itself,
an oscillating phenomenon in the leading term of the asymptotic formula.
When the position lies outside of the support, one can establish an asymptotic formula of large deviation type.
The rate function, which expresses the exponential decay rate,  is explicitly given.
Around the boundary of the support of the limit distribution (called the `wall'),
the asymptotic formula is described in terms of the Airy function.
\end{abstract}
\maketitle
\section{Introduction}
\label{Intro}
\setcounter{equation}{0}
The notion of discrete time quantum random walks, often called {\it quantum walks}, was
introduced by Aharonov-Davidovich-Zagury \cite{ADZ} in 1993 as a quantum counterpart of
the classical $1$-dimensional random walks. Since then, quantum walks (and generalizations) have been investigated in various contexts
from both sides of mathematics and physics. See \cite{Ke}, \cite{Ko2} for the
historical backgrounds and the recent developments.
In this paper, we shall give various asymptotic formulae for the transition probability associated with
the $1$-dimensional quantum walks.
Our formulae explain the peculiar phenomena of the probability distribution on the whole interval $[-1,1]$ 
(Figure \ref{fig:prob}\footnote{These figures are due to Dr. Takuya Machida, and who kindly allows 
us to use these pictures.})
which is quite a bit different from the case of classical walks.
\begin{figure}[htbp]
\begin{center}
\includegraphics[width=.55\linewidth]{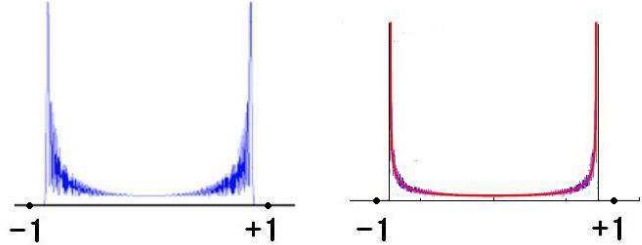}
\end{center}
\caption{Probability disctribution and its weak-limit distribution}
\label{fig:prob}
\end{figure}

The quantum walk we consider is concerned with a particle with spin $1/2$ whose position is restricted to $\mb{Z}$. 
To explain the set-up, we consider the Hilbert space
$$
\ell^{2}(\mb{Z},\mb{C}^{2})=\{f:\mb{Z}\longrightarrow \mb{C}^2\,;\,\|f\|^{2}_{\ell}:=\sum_{y \in \mb{Z}}\|f(y)\|^{2}<\infty\}
$$
with the inner product defined by
\[
\ispa{f,g}_{\ell}:=\sum_{x \in \mb{Z}}\ispa{f(x),g(x)},\quad f,g \in \ell^{2}(\mb{Z},\mb{C}^{2}),
\]
where $\|\cdot\|$ and $\langle\cdot,\cdot\rangle$ are the standard norm and inner product on $\mb{C}^{2}$, respectively.
For $x \in \mb{Z}$ and $u \in \mb{C}^{2}$, define $\delta_{x} \otimes u\in \ell^{2}(\mb{Z},\mb{C}^{2})$ by
\[
(\delta_{x} \otimes u) (y)=
\begin{cases}
u & (y = x),\\
0 & (y \neq x).
\end{cases}
\]
Then, the set $\{\delta_{x} \otimes e_{1},\delta_{x} \otimes e_{2}\}_{x \in \mb{Z}}$, where $\{e_{1},e_{2}\}$ denotes the
standard basis in $\mb{C}^{2}$, forms an orthonormal basis of $\ell^{2}(\mb{Z},\mb{C}^{2})$.
One has
\[
\|f(y)\|^{2}=|\ispa{f(y),e_{1}}|^{2}+|\ispa{f(y),e_{2}}|^{2}=|\ispa{f,\delta_{y} \otimes e_{1}}_{\ell}|^{2}+|\ispa{f,\delta_{y} \otimes e_{2}}_{\ell}|^{2}.
\]
Let $A \in SU(2)$, and write
\begin{equation}\label{matrixA}
A=
\begin{pmatrix}
a & b \\
-\ol{b} & \ol{a}
\end{pmatrix},\quad
a,b \in \mb{C},\quad |a|^{2}+|b|^{2}=1.
\end{equation}
Decompose the matrix $A$ as\footnote{Some of authors (as in \cite{Ke}, \cite{Ko1}) use the decomposition of the matrix $A$ by the row vectors
instead of the decomposition $\eqref{deco}$. See below. }
\begin{equation}\label{deco}
A=P+Q,\quad
P=
\begin{pmatrix}
a & 0 \\
-\ol{b} & 0
\end{pmatrix},\quad
Q=
\begin{pmatrix}
0 & b \\
0 & \ol{a}
\end{pmatrix}.
\end{equation}
A {\it quantum walk} is described by a unitary operator $U:\ell^{2}(\mb{Z},\mb{C}^{2})\longrightarrow \ell^{2}(\mb{Z},\mb{C}^{2})$ defined by
\begin{equation}\label{qwalk}
(Uf)(x)=P(f(x+1))+Q(f(x-1)),\quad f \in \ell^{2}(\mb{Z},\mb{C}^{2}),\quad x \in \mb{Z}.
\end{equation}
For a vector $\varphi \in \mb{C}^{2}$ with $\|\varphi\|=1$, the quantity
\begin{equation}\label{Ddist}
p_{n}(\varphi;x,y)=\|U^{n}(\delta_{x} \otimes \varphi)(y)\|^{2},
\end{equation}
represents the transition probability that the particle with the initial state $\delta_x\otimes \varphi$ is found at $y$ after the $n$-step evolution. 
The unitarity of the operator $U$ and the assumption that $\|\varphi\|=1$
tell that the sequence $\{p_{n}(\varphi;x,y)\}_{y \in \mb{Z}}$ is a probability distribution on the lattice $\mb{Z}$.
As in the case of classical random walks on $\mb{Z}$, the transition probability has the property 
\begin{equation}\label{property}
p_{n}(\varphi;x,y)=p_{n}(\varphi;0,y-x),\quad x,y \in \mb{Z}.
\end{equation}
Thus it is enough to consider the case where $x$ is the origin. Furthermore one can easily check that $p_{n}(\varphi;0,y)=0$
when $n+y$ is odd or $|y|>n$. From now on, we set $p_{n}(\varphi;y):=p_{n}(\varphi;0,y)$.
Our aim is to give asymptotic formulae for $p_{n}(\varphi;y)$ as $n$ tends to infinity.
When $a=0$ or $b=0$, the distribution $p_{n}(\varphi;x)$ can be easily computed.
In fact, 
\begin{enumerate}
\item if $a=0$, we have $p_{2n}(\varphi;x)=\delta_{x,0}$ and $p_{2n+1}(\varphi;x)=|\varphi_{1}|^{2}\delta_{x,-1}+|\varphi_{2}|^{2}\delta_{x,1}$, and
\item if $b=0$, we have $p_{n}(\varphi;x)=|\varphi_{1}|^{2}\delta_{x,-n}+|\varphi_{2}|^{2}\delta_{x,n}$,
\end{enumerate}
where $\varphi=\,\!^{t}(\varphi_{1},\varphi_{2})$.
Thus we only consider the case where $ab \neq 0$. 

To state our main theorems, we define the function $\lambda_{A}(\varphi)$ for $\varphi \in \mb{C}^{2}$ ($\|\varphi\|=1$) by
\begin{equation}\label{dist2}
\lambda_{A}(\varphi)=
|\varphi_{2}|^{2}-|\varphi_{1}|^{2}+\frac{1}{|a|^{2}}(ab\ol{\varphi}_{1}\varphi_{2}+\ol{a}\ol{b}\varphi_{1}\ol{\varphi}_{2}),\quad
\varphi=\,\!^{t}(\varphi_{1},\varphi_{2}) \in \mb{C}^{2}.
\end{equation}
It is shown in \cite{Ko1} that the probability measure,
\begin{equation}\label{Dmeas}
dm_{n,\varphi}=\sum_{y \in \mb{Z}}p_{n}(\varphi;y)\delta_{y/n},
\end{equation}
converges weakly to the one-dimensional distribution,
\begin{equation}\label{Lmeas}
\chi_{(-|a|,|a|)}\rho(\xi)\,d\xi,\quad \rho(\xi)=\frac{|b|(1+ \lambda_{A}(\varphi)\xi)}{\pi (1-\xi^{2})\sqrt{|a|^{2}-\xi^{2}}},
\end{equation}
where $\delta_{y/n}$ is the Dirac measure at $y/n$, and
$\chi_{(-|a|,|a|)}$ is the characteristic function of the open interval $(-|a|,|a|)$.
Roughly speaking, our asymptotic formulae for $p_{n}(\varphi;y)$ heavily depend on
the normalized position $y/n$ according as
\begin{enumerate}
\item $y/n$ is inside the interval $(-|a|,|a|)$,
\item $y/n$ stays around the wall; say, $y/n \sim \pm |a|$, or
\item $y/n$ is outside the interval; say, $|y/n|>|a|$.
\end{enumerate}

It is interesting to point out that our formulae described below resemble the Plancherel-Rotach formula for the
Hermite functions (Theorem 8.22.9 in \cite{Sz}), which tells us that the Hermite functions are
\begin{enumerate}
\item asymptotically oscillating when the position is in the classically allowed region,
\item approximated by the Airy function around the wall of the potential, and
\item exponentially decaying in the classically hidden region.
\end{enumerate}
In this view,
\begin{enumerate}
\item the interval $(-|a|,|a|)$ is called the `allowed' region,
\item $\pm |a|$ is called the `wall', and
\item the outside of $(-|a|,|a|)$ is called the `hidden' region.
\end{enumerate}
The precise statement for the `allowed' region is stated as follows.
\begin{thm}\label{inner}
Let $\alpha$ be a real number with $0<\alpha<|a|$. Then, for the integers $y$ satisfying
\begin{equation}\label{orderAs1}
|\xi_{n}| \leq \alpha,\quad \xi_{n}=y/n, 
\end{equation}
we have
\begin{equation}\label{innerAF}
p_{n}(\varphi;y)=\frac{(1+(-1)^{n+y})|b|}{\pi n (1-\xi_{n}^{2})\sqrt{|a|^{2}-\xi_{n}^{2}}}
\left[
1+\lambda_{A}(\varphi)\xi_{n} +{\rm OSC}_{n}(\xi_{n})+O(1/n)
\right]
\end{equation}
as $n \to \infty$ uniformly in $y$ satisfying $\eqref{orderAs1}$. Here ${\rm OSC}_{n}(\xi)$ is a function of the form
$$A(\xi)\cos (n\theta(\xi))+B(\xi)\sin (n\theta(\xi)).
$$
\end{thm}

It is obvious that, for any $\xi \in (-|a|,|a|)$, one can construct a sequence of integers $\{y_{n}\}$ satisfying 
\begin{equation}\label{orderAs}
y_{n}=n\xi+O(1) \quad (n \to \infty),  
\end{equation} 
and the integers $y$ in Theorem $\ref{inner}$ can be replaced by such a sequence $\{y_{n}\}$. 
Furthermore, if we allow $\{n\}$ to be a sequence of positive integers, then we can construct a sequence $\{y_{n}\}$ satisfying 
\begin{equation}\label{orderAs2}
y_{n}=n\xi+O(1/n) \quad (n \to \infty), 
\end{equation}
by using the theory of continued fractions. 
For such a sequence satisfying $\eqref{orderAs2}$, all of $\xi_{n}$ in $\eqref{innerAF}$ can be replaced by $\xi$. 

The factor $1+(-1)^{n+y_{n}}$ in the right hand side of $\eqref{innerAF}$ comes
from the fact that $p_{n}(\varphi;y)$ vanishes when $n+y$ is odd.
The functions $A(\xi)$, $B(\xi)$ and $\theta(\xi)$ in ${\rm OSC}_{n}(\xi)$ are computable
(see Section \ref{allowedRG}, in particular $\eqref{OSCF}$). 

Observe that the distribution function $\rho(\xi)$ defined in (\ref{Lmeas}) appears also in the asymptotic formula $\eqref{innerAF}$ with the factor $1/n$.
In fact, from Theorem $\ref{inner}$, one can deduce the following corollary which is a special case of Theorem 1 in \cite{Ko1}. 

\begin{cor}\label{KonnoT}
Let $\alpha$, $\beta$ be real numbers with $-|a|<\alpha<\beta<|a|$. 
Then we have 
\begin{equation}\label{KonnoAF}
\lim_{n \to \infty}\sum_{y \in \mb{Z}\,;\,\alpha \leq y/n \leq \beta}p_{n}(\varphi;y)
=\int_{\alpha}^{\beta}\rho(\xi)\,d\xi, 
\end{equation}
where the function $\rho(\xi)$ is defined in $\eqref{Lmeas}$. 
\end{cor}

The distribution $p_{n}(\varphi;y)$ mainly concentrates on the interval $(-|a|n,|a|n)$ when $n$ is large enough.
However, $p_{n}(\varphi;y)$ can be positive even when $|y| \geq |a|n$. Hence it would be quite natural to
consider its asymptotic behavior for $\{y_n\}$ with $y_n \sim \pm n|a|$ or $|y_n|>n|a|$.

\begin{thm}\label{wall}
Suppose that a sequence of integers $\{y_{n}\}$ satisfies the following$:$
\begin{equation}\label{wallAs}
y_{n}=\pm n|a|+d_{n},\quad d_{n}=O(n^{1/3}).
\end{equation}
Then we have
\begin{equation}\label{wallAF}
p_{n}(\varphi;y_{n}) = (1+(-1)^{n+y_{n}})\alpha^{2} n^{-2/3}
\left|
\ai\left(\pm \alpha n^{-1/3}d_{n}\right)
\right|^{2}
\left(
1\pm |a|\lambda_{A}(\varphi)
\right)+O(1/n),
\end{equation}
where $\alpha=(2/|a||b|^{2})^{1/3}$ and  $\ai(x)$ is the Airy function.
\end{thm}
In the formula $\eqref{wallAF}$, the variable in the Airy function is bounded, and hence its leading term is of order $O(n^{-2/3})$.

The asymptotic formula in the `hidden' region takes the following form.

\begin{thm}\label{hidden}
Let $\xi \in \mb{R}$ satisfy $|a|<|\xi|<1$. Suppose that a sequence of
integers $\{y_{n}\}$ satisfies $\eqref{orderAs}$. 
Then we have
\begin{equation}\label{hiddenAF}
p_{n}(\varphi;y_{n})=\frac{(1+(-1)^{n+y_{n}})|b|}{\pi n (1-\xi^{2})\sqrt{\xi^{2}-|a|^{2}}}e^{-nH_{Q}(\xi_{n})}(G(\xi)+O(1/n)),
\end{equation}
where $\xi_{n}=y_{n}/n$, $G(\xi)$ is a smooth non-negative function in $|a|<|\xi|<1$,
and the function $H_{Q}(\xi)$, which is positive and convex in this region, is given by
\begin{equation}\label{rateQ}
\begin{split}
H_{Q}(\xi) & =2|\xi|\log\left(
|b||\xi|+\sqrt{\xi^{2}-|a|^{2}}
\right)-2\log
\left(
|b|+\sqrt{\xi^{2}-|a|^{2}}
\right) \\
& \hspace{20pt}+
(1-|\xi|)\log\left(1-\xi^{2}\right)
-2|\xi|\log |a|.
\end{split}
\end{equation}
\end{thm}
This asymptotic formula tells that a large deviation property with the rate function $H_{Q}(\xi)$ holds in the `hidden' region.
Indeed, we have the following as a direct consequence of Theorem $\ref{hidden}$.
\begin{cor}\label{largeD}
Under the same assumption as in Theorem $\ref{hidden}$, we have
\begin{equation}\label{largeDAF}
\lim_{n \to \infty}\frac{1}{n}\log p_{n}(\varphi;y_{n})=-H_{Q}(\xi),
\end{equation}
where $n$ runs over positive integers such that $n+y_n$ is even.
\end{cor}
Here is a remark. As seen above, we have defined the unitary operator $U$
by employing the decomposition $\eqref{deco}$ in terms of the column vectors of $A$. Some authors use the decomposition
\[
A=R+S,\quad
R=
\begin{pmatrix}
a & b \\
0 & 0
\end{pmatrix},\quad
S=
\begin{pmatrix}
0 & 0 \\
-\ol{b} & \ol{a}
\end{pmatrix}.
\]
to define the unitary operator
\[
(Vf)(x):=R(f(x+1))+S(f(x-1)),\quad f \in \ell^{2}(\mb{Z},\mb{C}^{2}),\quad x \in \mb{Z},
\]
and the associated transition probability
\[
q_{n}(\psi;x)=\|V^{n}(\delta_{0} \otimes \psi)(x)\|^{2},\quad \psi \in \mb{C}^{2},\ \|\psi\|=1.
\]
It is easy to see that $VA^{*}=A^{*}U$ on $\ell^{2}(\mb{Z},\mb{C}^{2})$, and from this we have
\[
q_{n}(\psi;x)=p_{n}(A\psi;x),\quad x \in \mb{Z},\quad \psi \in \mb{C}^{2},\ \|\psi\|=1.
\]
Hence the use of different decompositions does not cause any significant differences in conclusion. 

We close Introduction by mentioning the strategy to prove Theorems $\ref{inner}$, $\ref{wall}$, $\ref{hidden}$ and the organization of the present paper.
First, we express the transition probability $p_{n}(\varphi;y)$ as the sum of the modulus square of
the transition amplitudes; say,
\begin{equation}\label{Pdist2}
p_{n}(\varphi;y)=|\ispa{U^{n}(\delta_{0} \otimes \varphi),\,\delta_{y} \otimes e_{1}}_{\ell}|^{2}+|\ispa{U^{n}(\delta_{0} \otimes \varphi),\,\delta_{y} \otimes e_{2}}_{\ell}|^{2}.
\end{equation}
Thus it is enough to find asymptotic formula
for the transition amplitudes $\ispa{U^{n}(\delta_{0} \otimes \varphi),\,\delta_{y} \otimes e_{i}}_{\ell}$ ($i=1,2$).
The starting point is the following integral formula
\begin{equation}\label{int00}
\begin{split}
\ispa{U^{n}(\delta_{0} \otimes \varphi),\,\delta_{y} \otimes \psi}_{\ell} & =
\frac{1}{2\pi i}\int_{C}z^{-y-1}\ispa{A(z)^{n}\varphi,\psi}\,dz \\
& = \frac{\omega^{-y}}{2\pi i}
\int_{C}z^{-y-1}\ispa{A(\omega z)^{n}\varphi,\psi}\,dz,
\end{split}
\end{equation}
where $\omega=a/|a|$ and $\psi \in \mb{C}^{2}$ with $\|\psi\|=1$. The contour $C$ is a simple closed path around the origin, and the matrix $A(z)$ is
defined by
\begin{equation}\label{sl2}
A(z):=Pz^{-1}+Qz=
\begin{pmatrix}
a z^{-1} & bz \\
-\ol{b}z^{-1} & \ol{a}z
\end{pmatrix}
=
\begin{pmatrix}
a & b \\
-\ol{b} & \ol{a}
\end{pmatrix}
\begin{pmatrix}
z^{-1} & 0 \\
0 & z
\end{pmatrix}.
\end{equation}
The integral formula $\eqref{int00}$ was also used in \cite{GJS} to obtain the
weak limit of the probability measure $\eqref{Dmeas}$ by considering the $r$-th moment of $\eqref{Dmeas}$.
Our strategy is to analyze the integral $\eqref{int00}$ by using a refined method of stationary phase due to H\"{o}rmander (Theorem 7.7.5 in \cite{Ho}).
Indeed, applying this method to the case where the contour $C$ is the unit circle, 
one can prove Theorem $\ref{inner}$ (Section $\ref{allowedRG}$). 
The proof of Corollary $\ref{KonnoT}$ involves a property of certain exponential sum determined by the 
asymptotic formula $\eqref{innerAF}$ (Lemma $\ref{RmLb}$). Thus, we give the details of the proof of Corollary $\ref{KonnoT}$ in Section $\ref{allowedRG}$. 
When $y/n$ stays around the wall, the phase function for this integral has degenerate critical points.
But the third derivatives at the critical points do not vanish, and we can use the argument given in the proof of Theorem 7.7.18 in \cite{Ho}
to prove Theorem $\ref{wall}$ (Section $\ref{aroundW}$).
Finally, we change the contour $C$ in the integral $\eqref{int00}$ to pick up a suitable critical point of the phase function and
apply again the refined method of stationary phase. A detail of the proof of Theorem $\ref{hidden}$ is given in Section $\ref{hiddenRG}$.

\section{Asymptotics in the allowed region}
\label{allowedRG}
\setcounter{equation}{0}
This section gives a proof of Theorem \ref{inner}.
To analyze the integral $\eqref{int00}$, we use a diagonalization of the matrix $A(\omega z)$.
The characteristic equation of the matrix $A(\omega z)$ is given by
\begin{equation}\label{eigenE}
\lambda^{2}-2|a|\phi(z)\lambda +1=0,
\end{equation}
where, for $z \in \mb{C} \setminus \{0\}$, we set $\phi(z)=(z+z^{-1})/2$. Thus the matrix $A(\omega z)$ is
diagonalizable except when the parameter $z$ satisfies $\phi(z)=\pm 1/|a|$, or equivalently, when
\begin{equation}\label{except}
z=\frac{1\pm |b|}{|a|},\ -\frac{1 \pm |b|}{|a|}.
\end{equation}
Define the function $\lambda(z)$ by
\begin{equation}\label{eigen1}
\lambda(z)=|a|\phi(z)+i\sqrt{1-|a|^{2}\phi(z)^{2}},
\end{equation}
which is holomorphic on the domain
\begin{equation}\label{Dom}
D=\mb{C} \setminus \{x \in \mb{R}\,;\,|x| \leq (1-|b|)/|a|\ \mbox{or}\ |x| \geq (1+|b|)/|a|\}.
\end{equation}
Then, for $z \in D$, the eigenvalues of $A(\omega z)$ are $\lambda(z)$ and $\lambda(z)^{-1}$, where 
\begin{equation}\label{eigen2}
\lambda(z)^{-1}=|a|\phi(z)-i\sqrt{1-|a|^{2}\phi(z)^{2}}.
\end{equation}
From $\eqref{eigen1}$, $\eqref{eigen2}$, we have $\lambda(-z)=-\lambda(z)^{-1}$ for $z \in D$.
Since $(1-|b|)/|a|<1<(1+|b|)/|a|$, the function $\lambda(z)$ is holomorphic around the unit circle.
In the rest of this section, we take the unit circle $S^{1}$ for the contour $C$ in the integral $\eqref{int00}$.
Note that $\lambda(z) \in S^{1}$ for each $z \in S^{1}$. Define the unit vector $u(z) \in \mb{C}^{2}$ for $z \in S^{1}$ by
\[
u(z)=\frac{1}{p(z)}
\begin{pmatrix}
\frac{ab}{|a|}z \\
\lambda(z)-|a|\ol{z}
\end{pmatrix},\quad
p(z)=\sqrt{|b|^{2}+|\lambda(z)-|a|\ol{z}|^{2}}.
\]
For $\psi \in \mb{C}^{2}$ with $\|\psi\|=1$, define a vector $J\psi \in \mb{C}^{2}$ by
\[
J\psi=
\begin{pmatrix}
-\ol{\psi}_{2} \\
\ol{\psi}_{1}
\end{pmatrix},\quad
\psi=
\begin{pmatrix}
\psi_{1} \\
\psi_{2}
\end{pmatrix}.
\]
The matrix $[\psi,J\psi]$ formed by two column vectors $\psi$, $J\psi$ is in $SU(2)$. We may easily check
\[
A(\omega z)u(z)=\lambda(z) u(z),\quad A(\omega z)Ju(z)=\ol{\lambda(z)} Ju(z).
\]
Thus any $\varphi \in \mb{C}^{2}$ is represented as $\varphi=\ispa{\varphi,u(z)}u(z)+\ispa{\varphi,Ju(z)}Ju(z)$, and
hence the integral formula $\eqref{int00}$ becomes
\begin{equation}\label{int01}
\ispa{U^{n}(\delta_{0} \otimes \varphi), \delta_{y} \otimes \psi}_{\ell}  =\frac{\omega^{-y}}{2\pi i}
\int_{|z|=1}z^{-y-1}\lambda(z)^{n}f_{\psi}(z)\,dz
+ \frac{\omega^{-y}}{2\pi i}\int_{|z|=1}
z^{-y-1}\ol{\lambda(z)}^{n}g_{\psi}(z)\,dz,
\end{equation}
where we set
\begin{equation}\label{FGpsi}
f_{\psi}(z)=\ispa{\varphi,u(z)}\ispa{u(z),\psi},\quad
g_{\psi}(z)=\ispa{\varphi,Ju(z)}\ispa{Ju(z),\psi}.
\end{equation}

\begin{lem}\label{GandF}
For any $z \in S^{1}$, we have $f_{\psi}(-z)=g_{\psi}(z)$.
\end{lem}

\begin{proof}
The vector $u(-z)$ is an eigenvector of $A(-\omega z)=-A(\omega z)$ with the eigenvalue $\lambda(-z)$.
Note that, for $z \in S^{1}$, we have $\lambda(-z)=-\lambda(z)^{-1}=-\ol{\lambda(z)}$.
Thus we see that $A(\omega z)u(-z)=\ol{\lambda(z)}u(-z)$, which
means that $u(-z)$ is an eigenvector of $A(\omega z)$ with the eigenvalue $\ol{\lambda(z)}$.
But, $Ju(z)$ is also an eigenvector with the same eigenvalue $\ol{\lambda(z)}$.
Hence, there exists a constant $c(z) \in S^{1}$ such that $u(-z)=c(z)Ju(z)$, and we get
\[
f_{\psi}(-z)=\ispa{\varphi,c(z)Ju(z)}\ispa{c(z)Ju(z),\psi}=c(z)\ol{c(z)}\ispa{\varphi,Ju(z)}\ispa{Ju(z),\psi}=g_{\psi}(z).
\]
\end{proof}

On the unit circle, the function $\lambda(z)$ can be written as
\begin{equation}\label{eigen3}
\lambda(e^{it})=|a|\cos t +i\sqrt{1-|a|^{2}\cos^{2}t}=e^{i\mu(t)},\quad t \in \mb{R},
\end{equation}
where we set
\begin{equation}\label{Fmu}
\mu(t)={\rm Cos}^{-1}(|a|\cos t).
\end{equation}
Note that $0 \leq \mu(t) \leq \pi$. Define 
\begin{equation}\label{int03}
J(n,y)  =J(\psi,n,y):=\frac{1}{2\pi i}\int_{|z|=1}z^{-y-1}\lambda(z)^{n}f_{\psi}(z)\,dz
= \frac{1}{2\pi}\int_{-\pi}^{\pi}e^{-iyt+in\mu(t)}f_{\psi}(e^{it})\,dt.
\end{equation}
Then, from Lemma $\ref{GandF}$, it follows that
\begin{equation}\label{int04}
\ispa{U^{n}(\delta_{0} \otimes \varphi), \delta_{y} \otimes \psi}_{\ell}
=\omega^{-y}(1+(-1)^{n+y})J(\psi,n,y).
\end{equation}
Theorem $\ref{inner}$ will be deduced easily from the following proposition. 

\begin{prop}\label{amp1}
Define the function $t(\eta)$ on the open interval $(-|a|,|a|)$ by
\begin{equation}\label{tF}
t(\eta)={\rm Sin}^{-1}
\left(
\frac{|b|\eta}{|a|\sqrt{1-\eta^{2}}}
\right),\quad |\eta|<|a|.
\end{equation}
Let $\alpha$ be a real number with $0<\alpha<|a|$. Then, for the integers $y$ satisfying $\eqref{orderAs1}$, 
we have 
\begin{equation}\label{amp1AF}
\begin{split}
& \ispa{U^{n}(\delta_{0} \otimes \varphi), \delta_{y} \otimes \psi}_{\ell} =
(1+(-1)^{n+y})\omega^{-y}
\sqrt{
\frac{|b|}{2\pi n(1-\xi_{n}^{2})\sqrt{|a|^{2}-\xi_{n}^{2}}}
} \\
& \hspace{15pt}
\times \left\{
e^{in[\mu(t(\xi_{n}))-\xi_{n} t(\xi_{n})]+\pi i/4}f_{\psi}(e^{it(\xi_{n})})+e^{-in[\mu(t(\xi_{n}))-\xi_{n} t(\xi_{n})]-\pi i/4}g_{\psi}(e^{-it(\xi_{n})})+O(1/n)
\right\}
\end{split}
\end{equation}
uniformly in $y$ satisfying $\eqref{orderAs1}$. 
\end{prop}

\begin{proof}
For $\eta$ with $-|a|<\eta<|a|$, we set
\[
t_{1}(\eta)=
\begin{cases}
\pi -t(\eta) & (\eta \geq 0), \\
-\pi -t(\eta) & (\eta<0).
\end{cases}
\]
Note that, for $\eta \in (-|a|,|a|)$, one has $|t(\eta)| < \pi/2$.
Let $\ve$ be a small positive number such that $\ve<\min\{(\pi-2t(\alpha))/6,\, (|a|-\alpha)/2,\,\pi/8\}$. 
Take a function $\chi(t) \in C_{0}^{\infty}(-\pi/4,\pi/4)$ such that $\chi(t)=1$ for $|t| \leq \ve$
and $\chi(t)=0$ for $|t| \geq 2\ve$.
We set $\chi_{n}(t)=1-\chi(t-t(\xi_{n}))-\chi(t-t_{1}(\xi_{n}))$, where, as in $\eqref{orderAs1}$, we set $\xi_{n}=y/n$. 
Note that $t(\eta)=t_{1}(\eta)$ if and only if $\eta=\pm |a|$.
The function $\chi(t)$ is chosen so that $\chi_{n}(t)=0$ near $t=t(\xi_{n})$ and $t=t_{1}(\xi_{n})$. 
The integral $\eqref{int03}$ can be written in the form
\[
J(\psi,n,y)=e^{-iyt(\xi_{n})}J_{0}(n)+(-1)^{n+y}e^{iyt(\xi_{n})}J_{1}(n)+R(n),
\]
where $J_{0}(n)$, $J_{1}(n)$ and $R(n)$ are given by
\[
\begin{split}
J_{0}(n) & =\frac{1}{2\pi}\int_{-\pi}^{\pi}
e^{in[\mu(t+t(\xi_{n}))-\xi_{n} t]}f_{\psi}(e^{i(t+t(\xi_{n}))})\chi(t)\,dt, \\
J_{1}(n) & =\frac{1}{2\pi}\int_{-\pi}^{\pi}
e^{-in[\mu(t-t(\xi_{n}))+\xi_{n} t]}g_{\psi}(e^{i(t-t(\xi_{n}))})\chi(t)\,dt, \\
R(n) & = \frac{1}{2\pi}\int_{t(\xi_{n})}^{t(\xi_{n})+2\pi}
e^{in[\mu(t)-\xi_{n} t]}f_{\psi}(e^{it})\chi_{n}(t)\,dt.
\end{split}
\]
Note that, in the above, we have used
\[
\mu(t+t_{1}(\eta))=\pi -\mu(t-t(\eta)),\quad f_{\psi}(e^{i(t+t_{1}(\eta))})=g_{\psi}(e^{i(t-t(\eta))}),
\]
which is easily shown by using Lemma $\ref{GandF}$ and the definition of $\mu(t)$. 
To use the method of stationary phase, we need the first and the second derivatives of $\mu(t)$, which are given by
\begin{equation}\label{mu12}
\mu'(t)=\frac{|a|\sin t}{\sqrt{1-|a|^{2}\cos^{2}t}},\quad
\mu''(t)=\frac{|a||b|^{2}\cos t}{(1-|a|^{2}\cos^{2}t)^{3/2}}.
\end{equation}
From $\eqref{mu12}$, $\mu'(t)=\eta$ if and only if $t=t(\eta)$ or $t=t_{1}(\eta)$. 
Let $c_{0}>0$ be the minimum of the function $|\mu'(t)-\eta|$ on the complement of the set 
\[
\{(\eta,t)\,;\,|t-t(\eta)| < \ve,\,|\eta| < \alpha+\ve\} \cup
\{(\eta,t)\,;\,|t-t_{1}(\eta)| < \ve,\,|\eta| < \alpha+\ve\}  
\]
in the rectangle $[-|a|,|a|] \times [-\pi,\pi]$. 
Then we have $|\mu'(t)-\xi_{n}| \geq c_{0}$ near the support of $\chi_{n}(t)$. 
In the integral $R(n)$, we extend $\chi_{n}(t)$ as a function on $\mb{R}$ 
so as to be zero for $t \geq t(\xi_{n})+2\pi$ and for $t \leq t(\xi_{n})$. 
Then it is still smooth depending on the parameter $n$ and $y$. 
The support of the function $\chi_{n}(t)$ on $\mb{R}$, so obtained, is contained in the bounded interval $[-t(\alpha),t(\alpha)+2\pi]$. 
Note that, each higher derivative of $\chi_{n}(t)$ is bounded uniformly in $n$ and $y$, 
and hence we see $R(n)=O(n^{-\infty})$ by Theorem 7.7.1 in \cite{Ho}. 
Since $\mu'(-t)=-\mu'(t)$, it is clear that the functions $\mu(t \pm t(\xi)) \mp \xi t$
have critical points only at $t=0$ in a neighborhood of the support of $\chi(t)$.
A direct computation gives
\[
\mu''(t(\eta))=\mu''(-t(\eta))=\frac{(1-\eta^{2})\sqrt{|a|^{2}-\eta^{2}}}{|b|},
\]
which shows that the critical point $t=0$ of the functions $\mu(t \pm t(\xi)) \mp \xi t$ is non-degenerate.
The derivatives of functions $f_{\psi}(e^{i(t+t(\xi_{n}))})\chi(t)$ and $g_{\psi}(e^{i(t-t(\xi_{n}))})\chi(t)$ 
are bounded uniformly in $n$, we can apply Theorem 7.7.5 in \cite{Ho} to obtain
\[
\begin{split}
J& (\psi,n,y) = \frac{1}{\sqrt{2\pi n \mu''(t(\xi_{n}))}} \\
& \times \left[
e^{in [\mu(t(\xi_{n}))-\xi_{n}t(\xi_{n})]+\pi i/4}f_{\psi}(e^{it(\xi_{n})})+(-1)^{n+y}
e^{-in [\mu(t(\xi_{n}))-\xi_{n}t(\xi_{n})]-\pi i/4}g_{\psi}(e^{-it(\xi_{n})})+O(1/n)
\right].
\end{split}
\]
Note that, in the above, the term $O(1/n)$ is uniform in $y$ satisfying $\eqref{orderAs1}$. 
Thus, multiplying the above by $\omega^{-y}(1+(-1)^{n+y})$ and simplifying the terms involving $(-1)^{n+y}$,
we get the formula $\eqref{amp1AF}$.
\end{proof}

\vspace{5pt}

\noindent{{\it Proof of Theorem $\ref{inner}$}}. \hspace{3pt}
Put $f_{i}(\eta)=f_{e_{i}}(e^{it(\eta)})$, $g_{i}(\eta)=g_{e_{i}}(e^{-it(\eta)})$ ($i=1,2$),
where $\{e_{1}, e_{2}\}$ is the standard basis in $\mb{C}^{2}$.
Define
\begin{equation}\label{OSC1}
z_{n}=e^{in[\mu(t(\xi_{n}))-\xi_{n} t(\xi_{n})]+i\pi/4},\quad \xi_{n}=y_{n}/n. 
\end{equation}
Taking the modulus square of $\eqref{amp1AF}$ for $\psi=e_{1}$ and $\psi=e_{2}$ and
summing up them, we get
\[
\begin{split}
p_{n}(\varphi;y) & = \frac{(1+(-1)^{n+y})|b|}{\pi n (1-\xi_{n}^{2})\sqrt{|a|^{2}-\xi_{n}^{2}}}
\left[
\left|
z_{n}f_{1}(\xi_{n})+\ol{z_{n}}g_{1}(\xi_{n})
\right|^{2}+
\left|
z_{n}f_{2}(\xi_{n})+\ol{z_{n}}g_{2}(\xi_{n})
\right|^{2}+O(1/n)
\right] \\
& = \frac{(1+(-1)^{n+y})|b|}{\pi n (1-\xi_{n}^{2})\sqrt{|a|^{2}-\xi_{n}^{2}}}
\left[
|f_{1}(\xi_{n})|^{2}+|f_{2}(\xi_{n})|^{2}+|g_{1}(\xi_{n})|^{2}+|g_{2}(\xi_{n})|^{2}+{\rm OSC}_{n}(\xi_{n})+O(1/n)
\right],
\end{split}
\]
where the function ${\rm OSC}_{n}(\eta)$ is given by 
\begin{equation}\label{OSCF}
\begin{split}
{\rm OSC}_{n}(\eta)& =e^{2in[\mu(t(\eta))-\eta t(\eta)]+i\pi/2}(f_{1}(\eta)\ol{g_{1}(\eta)}+f_{2}(\eta)\ol{g_{2}(\eta)}) \\
& \hspace{15pt}+e^{-2in[\mu(t(\eta))-\eta t(\eta)]-i\pi/2}(\ol{f_{1}(\eta)}g_{1}(\eta)+\ol{f_{2}(\eta)}g_{2}(\eta)).
\end{split}
\end{equation}
Now, a simple, but a little bit long computation gives
\[
|f_{1}(\eta)|^{2}+|f_{2}(\eta)|^{2}+|g_{1}(\eta)|^{2}+|g_{2}(\eta)|^{2}=1+\eta
\left(
|\varphi_{2}|^{2}-|\varphi_{1}|^{2}+\frac{1}{|a|^{2}}
(ab\ol{\varphi}_{1}\varphi_{2}+\ol{a}\ol{b}\varphi_{1}\ol{\varphi}_{2})
\right),
\]
which proves the asymptotic formula $\eqref{innerAF}$. \hfill$\square$

\begin{rem}
As stated in Introduction, the function ${\rm OSC}_{n}(\eta)$ can be represented in the form
\begin{equation}\label{OSCF2}
A(\eta)\cos (n\theta(\eta))+B(\eta)\sin (n\theta(\eta)),\quad \theta(\eta)=2\mu(t(\eta))-2\eta t(\eta).
\end{equation}
\end{rem}

\noindent{\it Proof of Corollary $\ref{KonnoT}$.} Let $\alpha,\beta$ be real numbers such that $-|a|<\alpha<\beta<|a|$. 
For each positive integer $n$, we write 
\[
\{y_{0},y_{1},\ldots,y_{s_{n}}\}=\{y \in \mb{Z}\,;\,y+n \in 2\mb{Z},\, n\alpha \leq y \leq n\beta\}\qquad (y_{j-1}<y_j).
\]
Then, we have $y_{j}-y_{j-1}=2$, $j=1,\ldots,s_{n}$.
Obviously $s_{n}=O(n)$. For simplicity, we set 
\begin{equation}\label{Rm10}
R(\eta)=\frac{|b|}{\pi (1-\eta^{2})\sqrt{|a|^{2}-\eta^{2}}}, 
\end{equation}
so that $\rho(\eta)=R(\eta)(1+\lambda_{A}(\varphi)\eta)$. 
By the formula $\eqref{innerAF}$, the sum in the left hand side of $\eqref{KonnoAF}$ is written as 
\[
\sum_{y \in \mb{Z}\,;\,\alpha \leq y/n \leq \beta}p_{n}(\varphi;y)
=\frac{2}{n}\sum_{j=0}^{s_{n}}
R(y_{j}/n)\left(
1+\lambda_{A}(\varphi)y_{j}/n
+{\rm OSC}_{n}(y_{j}/n)+O(1/n)
\right). 
\]
The sum of the terms in $O(1/n)$ is bounded because $s_{n}=O(n)$. 
Hence, because of the term $2/n$ in front of the sum, the contribution from these terms vanishes when $n$ tends to infinity. 
On the other hand, we find 
\[
\lim_{n \to \infty}\frac{2}{n}\sum_{j=0}^{s_{n}}R(y_{j}/n)(1+\lambda_{A}(\varphi)y_{j}/n)
=\lim_{n \to \infty}\frac{2}{n}\sum_{j=0}^{s_{n}}\rho(y_{j}/n)=\int_{\alpha}^{\beta}\rho(\eta)\,d\eta. 
\]
In order to examine the sum $\frac{2}{n}\sum_{j=0}^{s_{n}}
R(y_{j}/n){\rm OSC}_{n}(y_{j}/n)$, we write 
\begin{equation}\label{Rmaux1}
R(\eta){\rm OSC}_{n}(\eta)=K(\eta)e^{in\theta(\eta)}+\ol{K(\eta)}e^{-in\theta(\eta)}, 
\end{equation}
where $\theta(\eta)$ is defined in $\eqref{OSCF2}$ and $K(\eta)$ is a smooth function on the interval $(-|a|,|a|)$. 
Thus, Corollary $\ref{KonnoT}$ is a consequence of the following lemma. \hfill$\square$

\begin{lem}\label{RmL}
Let $K(\eta)$ be a smooth function on $[\alpha,\beta]$. Then we have
\begin{equation}\label{Rm11}
\lim_{n \to \infty}\frac{2}{n}\sum_{j=0}^{s_{n}}K(y_{j}/n)e^{in\theta(y_{j}/n)}=0. 
\end{equation}
\end{lem}

To prove Lemma $\ref{RmL}$, we first show the following. 
\begin{lem}\label{RmLb}
$(1)$ Suppose that $0$ is not contained in $[\alpha,\beta]$. Then there exists a constant $C>0$ such that 
\[
\left|
\sum_{j=0}^{k}e^{in \theta(y_{j}/n)}
\right|
\leq C
\]
for any $k$ and $n$ satisfying $0 \leq k \leq s_{n}$. 

$(2)$ Suppose that $0$ is contained in $[\alpha,\beta]$. 
Then there exists a constant $C>0$ such that 
\[
\left|
\sum_{j=0}^{k}e^{in \theta(y_{j}/n)}
\right|
\leq Cn^{1/2+\delta}
\]
for any $k$ and $n$ satisfying $0 \leq k \leq s_{n}$. 
\end{lem}
\begin{proof}
First, we prove (1). We only consider the case $\alpha >0$, because the case $\beta<0$ is handled similarly. 
The first and the second derivatives of $\theta(\eta)$ are given by 
\begin{equation}\label{DofT}
\theta'(\eta)=-2t(\eta),\quad \theta''(\eta)=-2\pi R(\eta), 
\end{equation}
where $t(\eta)$ and $R(\eta)$ are defined in $\eqref{tF}$ and $\eqref{Rm10}$, respectively. 
In particular, we have $-\pi<\theta'(\eta)<0$ on $[\alpha,\beta]$. 
The Taylor expansion gives us the equality
\[
e^{in\theta(y_{j+1}/n)}=e^{in\theta(y_{j}/n)+2i\theta'(y_{j}/n)+O(1/n)}, 
\]
where $O(1/n)$ is uniform with respect to $j$ with $0 \leq j \leq s_{n}-1$. 
From this, it follows that 
\[
e^{in\theta(y_{j}/n)}=\frac{e^{in\theta(y_{j+1}/n)}-e^{in\theta(y_{j}/n)}}{e^{2i\theta'(y_{j}/n)}-1}+O(1/n)
\quad (j=0,\ldots,s_{n}-1),  
\]
where we used the fact that the function $f(\eta):=1/(e^{2i\theta'(\eta)}-1)$ is bounded from above on $[\alpha,\beta]$. 
Using this formula and a partial summation, we have
\[
\sum_{j=0}^{k}e^{in \theta(y_{j}/n)}
=\sum_{j=1}^{k-1}e^{in\theta(y_{j}/n)}\left(
f(y_{j-1}/n)-f(y_{j}/n)
\right)+O(1),  
\]
where the term $O(1)$ is uniform in $k$ with $1 \leq k \leq s_{n}$. Since $f(\eta)$ is smooth on $[\alpha,\beta]$, we have 
\[
|f(y_{j-1}/n)-f(y_{j}/n)| \leq C|y_{j-1}-y_{j}|/n \leq C/n. 
\]
Hence 
\[
\left|
\sum_{j=0}^{k}e^{in \theta(y_{j}/n)}
\right|
\leq Ck/n \leq Cs_{n}/n \leq C, 
\]
which shows (1).

Next, let us prove (2). 
Note that we have $\theta'(0)=0$. The function $\theta(\eta)$ is written as
\[
\theta(\eta)=\theta(0)-\eta^{2}A(\eta),\quad A(\eta)=-\int_{0}^{1}(1-\lambda)\theta''(\lambda \eta)\,d\lambda. 
\]
Since $\theta''(\eta)<0$, we have $A(\eta)>0$. We set $f(\eta)=\eta \sqrt{A(\eta)}$. Then, $f'(0)=\sqrt{A(0)}>0$, and hence 
one can choose $\ve>0$ such that $f'(\eta)>0$ on $(-\ve,\ve)$. 
Furthermore, since $f(0)=0$, $\ve$ can be chosen so that $|2f(x)f'(x)|<\pi/8$ on $(-\ve,\ve)$. 
We assume first that the given interval $[\alpha,\beta] (\ni 0)$ is contained in $(-\ve,\ve)$. 
We write 
\[
\sum_{j=0}^{k}e^{in\theta(y_{j}/n)}=
\frac{n}{2}e^{in\theta(0)}\sum_{j=0}^{k-1}e^{-i(n^{1/2}f(y_{j}/n))^{2}}\left(\frac{y_{j+1}}{n}-\frac{y_{j}}{n}\right) +O(1). 
\]
We set $x_{j}=f(y_{j}/n)$. Since $x_{j+1}-x_{j}=O(1/n)$, we have 
\[
\frac{y_{j+1}}{n}-\frac{y_{j}}{n}=f^{-1}(x_{j+1})-f^{-1}(x_{j})=\frac{1}{f'(f^{-1}(x_{j}))}(x_{j+1}-x_{j})+O(1/n^{2}).  
\]
We further set $t_{j}=\sqrt{n}x_{j}$. Then, $t_{j+1}-t_{j}=O(1/\sqrt{n})$, and 
\begin{equation}\label{Rsum}
\sum_{j=0}^{k}e^{in\theta(y_{j}/n)}=
\frac{\sqrt{n}}{2}e^{in\theta(0)}R_{n}(k) +O(1),\quad R_{n}(k)=\sum_{j=0}^{k-1}L(t_{j}/\sqrt{n})e^{-it_{j}^{2}}(t_{j+1}-t_{j}),  
\end{equation}
where $L(x)=1/f'(f^{-1}(x))$. Now, we claim that, for arbitrary $\delta$ with $1/6<\delta<1/2$, we have 
\begin{equation}\label{Rmcl}
R_{n}(k)=O(n^{\delta})
\end{equation}
uniformly in $k$ and $n$. 
For simplicity, we only consider the case where $\alpha < 0 <\beta$, so that $t_{s_{n}} =\sqrt{n}x_{s_{n}} \to +\infty$ and 
$t_{0}=\sqrt{n}x_{0} \to -\infty$. 
Take a look at the sum 
\[
\tilde{S}_{n}(k):=\sum_{0\leq j \leq k-1\,;\,|t_{j}| \geq n^{\delta}}L(t_{j}/\sqrt{n})e^{-it_{j}^{2}}(t_{j+1}-t_{j}). 
\]
Note that when $|t_{j}| \leq n^{\delta}$, we have $|x_{j}|=|f(y_{j}/n)| \leq n^{\delta-1/2}$. 
Since $f(\eta)$ is an odd function, 
we see $|y_{j}/n| \leq f^{-1}(n^{\delta-1/2}) =O(n^{\delta-1/2})$. 
Let $l_{n}$ be an integer such that $y_{l_{n}}/n=(y_{0}+2l_{n})/n =O(1/n)$ as $n \to \infty$. 
Then we have
\[
|(2j-2l_{n})/n|-|(y_{0}+2l_{n})/n| \leq |y_{j}/n| \leq C n^{\delta-1/2}.
\]
Therefore, we have $|2j-2l_{n}| \leq Cn^{\delta+1/2}$, which shows that 
the number of summands in $R_{n}(k)-\tilde{S}_{n}(k)$ is of order $O(n^{\delta+1/2})$. 
Since $t_{j+1}-t_{j}=O(1/\sqrt{n})$, we obtain $R_{n}(k)=O(n^{\delta})+\tilde{S}_{n}(k)$. 
Thus, it is enough to show that $\tilde{S}_{n}(k)$ is of order $O(n^{\delta})$ uniformly in $k$. 
We shall only prove that the sum 
\[
S_{n}(k)=\sum_{0\leq j \leq k-1\,;\,t_{j} \geq n^{\delta}}L(t_{j}/\sqrt{n})e^{-it_{j}^{2}}(t_{j+1}-t_{j})
\]
is of order $O(n^{\delta})$, because, for the sum in $j$ with $t_{j} \leq -n^{\delta}$, the proof is carried out in the same way. 
We write 
\[
e^{-it_{j+1}^{2}}-e^{-it_{j}^{2}}=e^{-it_{j}^{2}}\left(
e^{-ia_{j}}-1
\right),\quad a_{j}=2t_{j}(t_{j+1}-t_{j})+(t_{j+1}-t_{j})^{2}=t_{j+1}^{2}-t_{j}^{2}. 
\]
Since $t_{j}=\sqrt{n}x_{j}=\sqrt{n}f(y_{j}/n)$, we see 
\[
t_{j}(t_{j+1}-t_{j})=nx_{j}(x_{j+1}-x_{j})=2f'(y_{j}/n)x_{j}+O(1/n)=2f(y_{j}/n)f'(y_{j}/n)+O(1/n). 
\]
Thus $a_{j}=4f(y_{j}/n)f'(y_{j}/n)+O(1/n)$. 
Let $j_{n}$ be the smallest number such that $t_{j_{n}} \geq n^{\delta}$. 
Since $0<2f(x)f'(x)<\pi/8$ on $(0,\ve)$, we have $0<a_{j}<\pi/4$ for $n$ large enough, 
and hence $e^{ia_{j}} \neq 1$ for any $j_{n} \leq j \leq s_{n}$. 
We shall check that 
\begin{equation}\label{Rm1001}
\frac{t_{j+1}-t_{j}}{e^{-ia_{j}}-1} =O(n^{-\delta}),\quad j=j_{n},\ldots,s_{n}. 
\end{equation}
Indeed, we have 
\[
|e^{-ia_{j}}-1| \geq \sin a_{j}=a_{j}A_{j},\quad A_{j}=\int_{0}^{1}\cos (a_{j}\lambda)\,d\lambda. 
\]
Since $0<a_{j}<\pi/4$, there exists a constant $c>0$ such that $A_{j}\geq c$. 
Thus, for $j_{n} \leq j \leq s_{n}$, 
\[
\left|
\frac{t_{j+1}-t_{j}}{e^{-ia_{j}}-1} 
\right|
\leq \frac{1}{c(t_{j+1}+t_{j})} \leq \frac{1}{2c}n^{-\delta}, 
\]
which shows $\eqref{Rm1001}$. 
Substituting now
\begin{equation}\label{exp11}
e^{-it_{j}^{2}}=\frac{e^{-it_{j+1}^{2}}-e^{-it_{j}^{2}}}{e^{-ia_{j}}-1}
\end{equation}
for $S_{n}(k)$ and using a partial summation, we get 
\[
S_{n}(k) = \sum_{j=j_{n}+1}^{k-1}(D_{j-1}-D_{j})e^{-it_{j}^{2}}+D_{k-1}e^{-it_{k}^{2}}+D_{j_{n}}e^{-it_{j_{n}}^{2}}, 
\]
where we set 
\begin{equation}\label{Rm100D}
D_{j}=\frac{t_{j+1}-t_{j}}{e^{-ia_{j}}-1}L(t_{j}/\sqrt{n}).  
\end{equation}
In view of $\eqref{Rm1001}$, we see 
\begin{equation}\label{Rm1002}
S_{n}(k)=\sum_{j=j_{n}+1}^{k-1}(D_{j-1}-D_{j})e^{-it_{j}^{2}} +O(n^{-\delta}). 
\end{equation}
Since $L(x)$ is a smooth function and since $t_{j+1}-t_{j}=O(1/\sqrt{n})$, we have
\begin{equation}\label{Rm1003}
D_{j-1}-D_{j} = L(t_{j}/\sqrt{n})
\left(
\frac{t_{j}-t_{j-1}}{e^{-ia_{j-1}}-1}-\frac{t_{j+1}-t_{j}}{e^{-ia_{j}}-1}
\right) +O(n^{-1-\delta}).
\end{equation}
To estimate the right hand side of $\eqref{Rm1003}$, we write 
\begin{equation}\label{Rm1004}
\frac{t_{j}-t_{j-1}}{e^{-ia_{j-1}}-1}-\frac{t_{j+1}-t_{j}}{e^{-ia_{j}}-1}
=\frac{t_{j}-t_{j-1}-(t_{j+1}-t_{j})}{e^{-ia_{j-1}}-1}
+(t_{j+1}-t_{j})
\left(
\frac{1}{e^{-ia_{j-1}}-1}-\frac{1}{e^{-ia_{j}}-1}
\right). 
\end{equation}
Since 
\begin{equation}\label{Rm1005}
\begin{split}
t_{j+1} & =\sqrt{n}f(y_{j+1}/n)=t_{j}+2f'(y_{j}/n)/\sqrt{n}+O(1/n^{3/2}), \\
t_{j-1} & =t_{j}-2f'(y_{j}/n)/\sqrt{n}+O(1/n^{3/2}), 
\end{split}
\end{equation}
we have $2t_{j}-t_{j+1}-t_{j-1}=O(1/n^{3/2})$. $\eqref{Rm1005}$ also shows that $t_{j+1}-t_{j} \geq c/\sqrt{n}$ with a constant $c>0$. 
Thus we have $1/(e^{-ia_{j}}-1)=O(n^{1/2-\delta})$, and hence the first term in the right hand side of $\eqref{Rm1004}$ is of order $O(n^{-1-\delta})$. 
Note 
\begin{equation}\label{Rm1007}
\frac{1}{e^{-ia_{j-1}}-1}-\frac{1}{e^{-ia_{j}}-1}=O(n^{-2\delta}). 
\end{equation}
Indeed we write 
\[
\frac{1}{e^{-ia_{j-1}}-1}-\frac{1}{e^{-ia_{j}}-1} 
=\frac{e^{-ia_{j-1}}}{(e^{-ia_{j-1}}-1)(e^{-ia_{j}}-1)}
\left(
e^{i(a_{j-1}-a_{j})}-1
\right). 
\]
Setting $h(x)=f(x)^{2}$, we have
\[
t_{j+1}^{2} =nh(y_{j+1}/n)=t_{j}^{2}+2h'(y_{j}/n)+O(1/n), \quad 
t_{j-1}^{2} = t_{j}^{2}-2h'(y_{j}/n)+O(1/n). 
\]
From this we have $a_{j}-a_{j-1}=t_{j+1}^{2}+t_{j-1}^{2}-2t_{j}^{2}=O(1/n)$. 
Since $1/(e^{-ia_{j}}-1)=O(n^{1/2-\delta})$, we obtain $\eqref{Rm1007}$. 
$\eqref{Rm1007}$ also shows that $D_{j}-D_{j-1}=O(n^{-1/2-2\delta})$. 
Therefore, by $\eqref{Rm1002}$, we have $S_{n}(k)=O(n^{1/2-2\delta})$. 
But since $\delta>1/6$, we obtain $S_{n}(k)=O(n^{\delta})$ as required. 

Next, we shall prove (2) for general $[\alpha,\beta]$ containing $0$. As before, we may suppose that $\alpha <0<\beta$. 
Take $c>0$ such that $c<\min\{|\alpha|,\ve,\beta\}$. We write 
\[
\sum_{j=0}^{k}e^{in\theta(y_{j}/n)}=\sum_{0 \leq j \leq k\,;\,y_{j}/n \in [\alpha,-c]}e^{in\theta(y_{j}/n)}+
\sum_{0 \leq j \leq k\,;\,y_{j}/n \in [c,\beta]}e^{in\theta(y_{j}/n)}+\sum_{0 \leq j \leq k\,;\,-c<y_{j}/n<c}e^{in\theta(y_{j}/n)}
\]
By (1), the first and the second sums are of order $O(1)$. 
Since each term of the sum is $O(1)$, we have 
\[
\sum_{0 \leq j \leq k\,;\,-c<y_{j}/n<c}e^{in\theta(y_{j}/n)}=\sum_{0 \leq j \leq k\,;\,-c\leq y_{j}/n \leq c}e^{in\theta(y_{j}/n)} +O(1). 
\]
By the fact we have just proved, the sum in the right hand side above is of order $O(n^{1/2+\delta})$, 
and hence we obtain the desired estimate. 
\end{proof}

\noindent{\it Proof of Lemma $\ref{RmL}$}
We set 
\[
T_{k}(n)=T_{k}(n;[\alpha,\beta]):=
\sum_{j=0}^{k}e^{in \theta(y_{j}/n)}. 
\]
If $0 \not\in [\alpha,\beta]$, then $T_{k}(n)=O(1)$ uniformly in $k$. 
which is, by Lemma $\ref{RmLb}$, of order $O(n^{\delta})$ for any $\delta$ with $1/6<\delta<1/2$ uniformly in $k=0,1,\ldots,s_{n}$. 
A partial summation leads us to 
\[
\frac{2}{n}\sum_{j=0}^{s_{n}}K(y_{j}/n)e^{in\theta(y_{j}/n)}
=\frac{2}{n}\sum_{j=1}^{s_{n}-1}T_{j}(n)\left(
K(y_{j}/n)-K(y_{j+1}/n)
\right)
+O(n^{-1+\delta})
\]
Since $K(\eta)$ is smooth, we have $K(y_{j}/n)-K(y_{j+1}/n)=O(1/n)$. 
Thus, Lemma $\ref{RmLb}$ shows that the first term in the above is of order $O(s_{n}n^{\delta-2})=O(n^{-1+\delta})$. 
Therefore, we conclude $\eqref{Rm11}$ \hfill$\square$

\section{Asymptotics around the wall}
\label{aroundW}
\setcounter{equation}{0}

In this section we give a proof of Theorem $\ref{wall}$.
As in the previous section, we take the unit circle $S^{1}$ for a contour $C$ in $\eqref{int00}$.
However, the phase functions $\mu(t)\mp |a|t$,
which appears in the integral $\eqref{int03}$, have degenerate critical points,
and thus we cannot use Theorem 7.7.5 in \cite{Ho}.
Instead one could use Theorem 7.7.18 in \cite{Ho} directly to obtain the following proposition.

\begin{prop}\label{amp2wall}
Suppose that a sequence of integers $\{y_{n}\}$ satisfies $\eqref{wallAs}$
Then we have
\begin{equation}\label{amp2AF}
\begin{split}
& \ispa{U^{n}(\delta_{0} \otimes \varphi),\,\delta_{y_{n}} \otimes \psi}_{\ell} \\
&\hspace{15pt} = \omega^{-y_{n}}e^{i\pi (n \mp y_{n})/2}(1+(-1)^{n+y_{n}}) f_{\psi}(\pm i) \alpha n^{-1/3}\ai(\pm \alpha n^{-1/3}d_{n})+O(n^{-2/3}).
\end{split}
\end{equation}
where $\alpha=(2/|a||b|^{2})^{1/3}$, and $\ai(x)$ is the Airy function defined by
\[
\ai(x)=\frac{1}{2\pi}\int_{-\infty}^{\infty}e^{i\xi^{3}/3+i\xi x}\,d\xi.
\]
\end{prop}

\begin{proof}
For simplicity, we consider only the case where $y_{n}=|a|n+d_{n}$ with the assumption $d_{n}=O(n^{1/3})$.
We need to analyze $J(\psi,n,y_{n})$ defined in $\eqref{int03}$.
For this sake, take a function $\rho(t) \in C_{0}^{\infty}(-\pi/4,\pi/4)$ satisfying $\rho(t)=1$ near $t=0$. 
Set $\rho_{1}(t)=1-\rho(t)$, and write
\[
J(\psi,n,y_{n})=J_{1}(n) +J_{2}(n),\quad
J_{1}(n)=\frac{e^{-i\pi y_{n}/2}}{2\pi}
\int_{-\pi/2}^{3\pi/2}e^{in \phi(t)+id_{n}t}\rho(t)f_{\psi}(ie^{-it})\,dt,
\]
where the integral $J_{2}(n)$ is defined by replacing $\rho(t)$ with $\rho_{1}(t)$ in the definition of $J_{1}(n)$, and
the function $\phi$ is given by
\[
\phi(t)=\mu(-t+\pi/2)+|a|t.
\]
The derivatives of $\phi$ up to the third order are given by
\[
\phi'(t)=-\frac{|a|\cos t}{(1-|a|^{2}\sin^{2}t)^{1/2}}+|a|,\ \
\phi''(t)=\frac{|a||b|^{2}\sin t}{(1-|a|^{2}\sin^{2}t)^{3/2}},\ \
\phi'''(t)=\frac{|a||b|^{2}\cos t (1+2|a|^{2}\sin^{2}t)}{(1-|a|^{2}\sin^{2}t)^{5/2}}.
\]
Thus, on the interval $[-\pi/2, 3\pi/2]$, $\phi'(t)=0$ if and only if $t=0$, and we have
\[
\phi(0)=\pi/2,\ \ \phi'(0)=\phi''(0)=0,\ \ \phi'''(0)=|a||b|^{2}>0.
\]
Since $\phi'(t) \neq 0$ on the support of $\rho_{1}(t)$, a standard argument involving
an integration by parts shows $J_{2}(n)=O(n^{-\infty})$.
One may also write
\[
\phi(t)=\pi/2 +t^{3}A(t),\quad A(t)=\frac{1}{2}\int_{0}^{1}(1-x)^{2}\phi'''(xt)\,dx.
\]
Note that $A(0)=|a||b|^{2}/6>0$. Set $B(t):=(3A(t))^{1/3}t$ on a small neighborhood of $t=0$.
Then we have
\begin{equation}\label{FB}
\phi(t)=\pi/2 +B(t)^{3}/3,\quad B(0)=0,\ \ B'(0)=\alpha^{-1}=(|a||b|^{2}/2)^{1/3}>0.
\end{equation}
Choosing $\rho(t)$ whose support is small enough, we may suppose that $B'(t)>0$ near the support of $\rho(t)$.
A change of variable leads us to
\begin{equation}\label{intJ1}
J_{1}(n)=\frac{e^{i\pi (n-y_{n})/2}}{2\pi}
\int e^{in(s^{3}/3+d_{n}B^{-1}(s)/n)}u(s)\,ds,\quad u(s)=\frac{\rho(B^{-1}(s))f_{\psi}(ie^{-iB^{-1}(s)})}{B'(B^{-1}(s))}.
\end{equation}
Taking $\rho(t)$ whose support is small enough if necessary, we may suppose that $u(s)$ is a compactly supported smooth function.
One can now apply Theorem 7.7.18 in \cite{Ho} directly to the integral $J_{1}(n)$ to finish the proof.
For the sake of completeness, we shall give a detail. 
Define a function $\varphi(s,r)$ by
\begin{equation}\label{auxphi}
\varphi(s,r)=s^{3}/3+rB^{-1}(s).
\end{equation}
Since
\[
\varphi(0,0)=(\partial_{s}\varphi)(0,0)=(\partial_{s}^{2}\varphi)(0,0)=0,\quad (\partial_{s}^{3}\varphi)(0,0)=2,
\]
we can apply Theorem 7.5.13 in \cite{Ho} which tells that there exists smooth real-valued functions $T(s,r)$, $a(r)$, $b(r)$ defined
on a neighborhood of $(s,r)=(0,0)$ satisfying
\begin{equation}\label{normal}
\begin{split}
a(0)& =b(0)=T(0,0)=0,\quad (\partial_{s}T)(0,0)>0,\\
\varphi& (s,r)=\frac{1}{3}T(s,r)^{3}+a(r)T(s,r)+b(r).
\end{split}
\end{equation}
By $\eqref{auxphi}$, the functions $T$, $a$, $b$ have the properties
\begin{equation}\label{prTAB}
T(s,0)=s, \quad a'(0)=\alpha,\quad b'(0)=0.
\end{equation}
Indeed, setting $r=0$ in the second line of $\eqref{normal}$, we have $T(s,0)=s$. Computing the derivative $(\partial_{r}\varphi)(s,0)$, we also get
\[
B^{-1}(s)=b'(0)+a'(0)s+s^{2}(\partial_{r}T)(s,0).
\]
Since $B^{-1}(0)=0$, $(B^{-1})'(0)=\alpha$, we have $\eqref{prTAB}$.
By the implicit function theorem, one can find a smooth function $s(t,r)$, defined on a neighborhood of $(t,r)=(0,0)$
such that $T(s(t,r),r)=t$. Since $T(s,0)=s$, we have $s(0,0)=0$, $(\partial_{t}s)(0,0)=1$, $(\partial_{t}^{2}s)(0,0)=0$.
A change of variables in the integral $\eqref{intJ1}$ gives
\begin{equation}\label{intJ11}
J_{1}(n)=\frac{e^{i \pi (n-y_{n})/2}}{2\pi}
\int e^{in(t^{3}/3+a(\beta_{n})t+b(\beta_{n}))}v(t,\beta_{n})\,dt,\quad
v(t,r)=u(s(t,r))(\partial_{t}s)(t,r),
\end{equation}
where we set $\beta_{n}=d_{n}/n$ which is of order $O(n^{-2/3})$.
Choosing $\rho(t)$ whose support is small enough,
one can assume that the function $v(t,r)$ is compactly supported around $(t,r)=(0,0)$.
Thus we can choose a function $\chi(t) \in C_{0}^{\infty}(\mb{R})$
such that $0 \leq \chi(t) \leq 1$ and $\chi(t)=1$ when $v(t,r) \neq 0$, and hence $v(t,r)=\chi(t)v(t,r)$.
The Malgrange preparation theorem (Theorem 7.5.6 in \cite{Ho}) tells that there are smooth functions $q(t,r)$, $c(r)$, $d(r)$ satisfying
\begin{equation}\label{Mal}
v(t,r)=q(t,r)(t^{2}+a(r))+c(r)t+d(r).
\end{equation}
Note that $v(0,0)=u(0)$, $(\partial_{t}v)(0,0)=u'(0)$ where the function $u(s)$ is defined in $\eqref{intJ1}$.
Thus we have
\begin{equation}\label{Mal2}
d(0)=u(0),\quad c(0)=u'(0).
\end{equation}
Inserting $\eqref{Mal}$ into $\eqref{intJ11}$ and integrating by parts, we obtain
\begin{equation}\label{intJ13}
\begin{split}
J_{1}(n) & = \frac{e^{i \pi (n-y_{n})/2}}{2\pi}e^{inb(\beta_{n})}
\left(
d(\beta_{n})\int e^{in(t^{3}/3+a(\beta_{n})t)}\chi(t)\,dt
\right. \\
& +\left. c(\beta_{n})\int e^{in(t^{3}/3+a(\beta_{n})t)}t \chi(t)\,dt
-\frac{1}{in}\int e^{in(t^{3}/3+a(\beta_{n})t)}\partial_{t}(\chi(t)q(t,\beta_{n}))\,dt
\right).
\end{split}
\end{equation}
The last term in $\eqref{intJ13}$ is of order $O(1/n)$.
For the first and the second integrals, a change of variables shows
\begin{equation}\label{aux111}
\begin{split}
& \int e^{in(t^{3}/3+a(\beta_{n})t)}\chi(t)\,dt \\
& =2\pi n^{-1/3}\ai(n^{2/3}a(\beta_{n}))
-n^{-1/3}\int e^{i(t^{3}/3+n^{2/3}a(\beta_{n})t)}(1-\chi(n^{-1/3}t))\,dt, \\
& \int e^{in(t^{3}/3+a(\beta_{n})t)}t\chi(t)\,dt \\
& = -2\pi i n^{-2/3}\ai'(n^{2/3}a(\beta_{n}))-n^{-2/3}
\int e^{i(t^{3}/3+n^{2/3}a(\beta_{n})t)}t(1-\chi(n^{-1/3}t))\,dt.
\end{split}
\end{equation}
By the assumption $\beta_{n}=d_{n}/n=O(n^{-2/3})$ and the properties $\eqref{normal}$, $\eqref{prTAB}$, we have
\begin{equation}\label{AiryE}
\ai(n^{2/3}a(\beta_{n}))=\ai(n^{2/3}\alpha \beta_{n})+O(n^{2/3}\beta_{n}^{2})=\ai(\alpha n^{-1/3}d_{n})+O(n^{-2/3}).
\end{equation}
By the same reason, the term $n^{-2/3}\ai'(n^{2/3}a(\beta_{n}))$ in the second line of $\eqref{aux111}$ is of order $O(n^{-2/3})$.
To estimate each integral in the right hand side of $\eqref{aux111}$, we need the following lemma.
\begin{lem}\label{error1}
Let $a_{n}$ be a sequence of real numbers such that $a_{n} \to 0$ as $n \to \infty$. 
Let $0<\delta<c$, $\ve>0$ be constants satisfying $\delta<1$, $\delta+\ve<c-\ve$, 
and $\chi(s)$ be a smooth function on $\mb{R}$ such that $\chi(s)=1$ for $|s| \leq \delta+\ve$ and $\chi(s)=0$ for $|s| \geq c-\ve$.
Then the integrals
\begin{equation}\label{errorR}
\int e^{in(s^{3}/3+a_{n}s)}(1-\chi(s))\,ds,\quad
\int e^{in(s^{3}/3+a_{n}s)}s(1-\chi(s))\,ds.
\end{equation}
are of order $O(n^{-\infty})$ as $n \to \infty$.
\end{lem}
Assuming Lemma $\ref{error1}$, one sees that the integrals in the right hand side of $\eqref{aux111}$ is of order $O(n^{-\infty})$.
Using $d(\beta_{n})=u(0)+O(n^{-2/3})$ and $nb(\beta_{n})=O(n\beta_{n}^{2})=O(n^{-1/3})$, we get
\[
\begin{split}
J_{1}(n) & =\frac{e^{i\pi (n-y_{n})/2}}{2\pi}e^{inb(\beta_{n})}
\left(
2\pi n^{-1/3}u(0)\ai(\alpha n^{-1/3}d_{n})+O(n^{-2/3})
\right) \\
& =e^{i\pi (n-y_{n})/2} n^{-1/3}u(0)\ai(\alpha n^{-1/3}d_{n})+O(n^{-2/3}).
\end{split}
\]
By the definition $\eqref{intJ1}$ of $u(s)$, we have $u(0)=\alpha f_{\psi}(i)$, and hence we conclude the proof.
\end{proof}

\vspace{10pt}

\noindent{{\it Proof of Lemma $\ref{error1}$}}\hspace{5pt}
For the integrals $\eqref{errorR}$, the integrals on the negative real line have the same form
as the integrals on the positive real line except the sign of $n$. Thus it is enough to
prove that the integrals
\[
R_{1}(n):=\int_{0}^{\infty} e^{in\psi_{n}(s)}(1-\chi(s))\,ds,\quad
R_{2}(n):=\int_{0}^{\infty} e^{in\psi_{n}(s)}s(1-\chi(s))\,ds
\]
are of order $O(n^{-\infty})$, where we set $\psi_{n}(s)=s^{3}/3+a_{n}s$.
We give a detail only for the integral $R_{2}(n)$ since the arguments for $R_{1}(n)$ is the same.
We may assume $|a_{n}|<\delta^{2}/2$ for every $n$.
Take $R>c$ and write $R_{2}(n)=S(n,R)+T(n,R)$, where
\[
S(n,R)=\int_{\delta}^{R}e^{in \psi_{n}(s)}s(1-\chi(s))\,ds,\quad
T(n,R)=\int_{R}^{\infty}e^{in \psi_{n}(s)}s\,ds.
\]
Since $\psi'(s)=s^{2}+a_{n} \neq 0$ for $\delta \leq s$, we can use the differential operator $\psi'(s)^{-1}(d/ds)$ to
perform an integration by parts, and get
\[
S(n,R)=\frac{Re^{in\psi_{n}(R)}}{in (R^{2}+a_{n})}-\frac{1}{in}\int_{\delta}^{R}e^{in \psi_{n}(s)}\frac{s\partial_{s}[(1-\chi(s))]}{s^{2}+a_{n}}\,ds
-\frac{1}{in}\int_{\delta}^{R}e^{in\psi_{n}(s)}(1-\chi(s))\frac{-s^{2}+a_{n}}{(s^{2}+a_{n})^{2}}\,ds,
\]
where we have used that $\chi(R)=0$ and $\chi(\delta)=1$.
By the assumptions of $\chi(t)$, the function $\partial_{s}(1-\chi(s))=-\chi'(s)$ has a compact support on
the interval $(\delta,R)$, and hence the integral in the second term is of order $O(n^{-\infty})$.
Repeating this procedure, we have, for any $N>0$, an expression of the form
\[
S(n,R)=\sum_{k=0}^{N-1}\frac{q_{k+1}(R)e^{in \psi_{n}(R)}}{(in)^{k+1}(R^{2}+a_{n})^{2k+1}}
+\frac{1}{(in)^{N}}\int_{\delta}^{R}
e^{in\psi_{n}(s)}(1-\chi(s))\frac{q_{N+1}(s)}{(s^{2}+a_{n})^{2N}}\,ds +O(n^{-\infty}),
\]
where $q_{k}(s)$ is a polynomial of degree at most $k$ independent of $R$. Note that the estimate $O(n^{-\infty})$
does not depend on $R$ (but depends on $N$). From this expression, it is easy to show that
\[
S(n,n^{N-1})=O(n^{-N}).
\]
Next, let us consider the integral $T(n,R)$. We set $f(\zeta)=\zeta e^{in(\zeta^{3}/3+a_{n}\zeta)}$, $\zeta \in \mb{C}$.
We take $L>0$ and $\eta>0$ such that $R<L$ and $0<\eta<c/2$.
Integrating $f(\zeta)$ on the rectangle determined by $R$, $L$, $L+i\eta$, $R+i\eta$ and letting $L \to \infty$, we have
\begin{equation}\label{error22}
T(n,R)=\int_{R}^{\infty}f(s+i\eta)\,ds +i\int_{0}^{\eta}f(R+it)\,dt.
\end{equation}
It is clear that
\[
|f(R+it)| \leq CRe^{-n(R^{2}t-t^{3}/3+a_{n}t)} \leq CRe^{-nR^{2}t/2},\quad 0 \leq t \leq \eta,
\]
where $C>0$ is a constant independent of $n$, $R$, and $R>0$ is chosen so that $R^{2}-\delta^{2} >2\eta^{2}$.
Thus the integral in the second term in $\eqref{error22}$ is of order $O(1/R)$. The first term in $\eqref{error22}$ is
estimated as
\[
\int_{R}^{\infty}|f(s+i\eta)|\,ds \leq C\int_{R}^{\infty}se^{-n\eta (s^{2}-\eta^{2}/3+a_{n})}\,ds \leq Cn^{-1}e^{-n\eta(R^{2}-\eta^{2}/3+a_{n})}.
\]
Thus, if we take $R=n^{k}$ with $k>1$, then the
first term in $\eqref{error22}$ is of order $O(n^{-1}e^{-\eta n^{2k+1}/2})$, and hence $T(n,n^{N-1})=O(n^{-N+1})$.
Therefore, we get
\[
R_{2}(n)=S(n,n^{N})+T(n,n^{N})=O(n^{-N})
\]
for any $N>0$, which gives the assertion.\hfill$\square$
\vspace{10pt}

\noindent{{\it Proof of Theorem $\ref{wall}$}.}\hspace{5pt}
Using $\eqref{FGpsi}$, we can easily check
\[
f_{e_{1}}(\pm i)=\frac{|b|^{2}}{2(1 \pm |a|)}\varphi_{1} \pm \frac{ab}{2|a|}\varphi_{2},\quad
f_{e_{2}}(\pm i)=\pm \frac{\ol{a}\ol{b}}{2|a|}\varphi_{1}+\frac{1 \pm |a|}{2}\varphi_{2},
\]
and hence we have
\[
|f_{e_{1}}(\pm i)|^{2}+|f_{e_{2}}(\pm i)|^{2}=\frac{1}{2}\pm \frac{|a|}{2}
\left(
|\varphi_{2}|^{2}-|\varphi_{1}|^{2}
+\frac{ab\ol{\varphi}_{1}\varphi_{2}+\ol{a}\ol{b}\varphi_{1}\ol{\varphi}_{2}}{|a|^{2}}
\right)
=\frac{1}{2}(1 \pm |a|\lambda_{A}(\varphi)).
\]
Taking the modulus square of $\eqref{amp2AF}$ for $\psi=e_{1},e_{2}$ and
summing up them,  we conclude $\eqref{wallAF}$.\hfill$\square$

\section{Asymptotis in the hidden region}
\label{hiddenRG}
\setcounter{equation}{0}

The purpose of this section is to prove Theorem $\ref{hidden}$. As in the previous sections,
we first handle the transition amplitude. 

\begin{prop}\label{amp3hidden}
Let $\xi$ satisfy $|a|<|\xi|<1$, and suppose that a sequence of integers $\{y_{n}\}$ satisfies $\eqref{orderAs}$.
Then we have
\begin{equation}\label{amp3AF}
\begin{split}
\ispa{U^{n}(\delta_{0} \otimes \varphi),\,\delta_{y_{n}}\otimes \psi}_{\ell}
& = (1+(-1)^{n+y_{n}})\omega^{-y_{n}}\frac{e^{-nH_{Q}(\xi_{n})/2}}{\sqrt{2\pi n}} \\
& \hspace{15pt} \times \sqrt{
\frac{|b|}{(1-\xi^{2})\sqrt{\xi^{2}-|a|^{2}}}
}
e^{\pi i (n-|y_{n}|)/2}(F_{\psi}(\xi)+O(1/n)),
\end{split}
\end{equation}
where $\xi_{n}=y_{n}/n$ and $F_{\psi}(\xi)$ is a smooth function in $|a|<|\xi|<1$ which is given in $\eqref{FpsiF}$, $\eqref{FpsiF2}$,
and $H_{Q}(\xi)$ is a convex positive-valued function defined in $\eqref{rateQ}$.
\end{prop}

In the previous sections, we have taken the unit circle for a contour $C$ in the formula $\eqref{int00}$.
The phase function, however, has no critical points when the parameter $\xi$ is outside the interval $(-|a|,|a|)$.
Therefore, it would be reasonable to change the contour $C$ to pick up a suitable critical point. 
Here a `phase function' means essentially the logarithm of the eigenvalues of the matrix $A(\omega z)$.

Before proceeding to the details, let us give a rough description of the phase function and its critical points.
Recall that, when $z \in \mb{C} \setminus 0$ does not coincide with the values in $\eqref{except}$, the matrix $A(\omega z)$ is diagonalizable.
Hence, by diagonalizing $A(\omega z)$, $\eqref{int00}$ is expressed as a linear combination of integrals of the form
\[
\int_{C}z^{-y-1}\lambda(z)^{n}f(z)\,dz,
\]
where $f(z)$ is holomorphic on a domain containing $C$, 
and $\lambda(z)$ is an eigenvalue of $A(\omega z)$ (we assume here that $\lambda(z)$ is also holomorphic). 
By expressing the path $C$ by a function $z=z(t)$ with $z'(t)\neq 0$, the integral above is written as
\[
\int e^{n \Phi(t)}f(z(t))\frac{z'(t)}{z(t)}\,dt,\quad \Phi(t)=\log \lambda(z(t)) -\eta \log z(t),
\]
where we set $\eta=y/n$. 
The critical point $c$ of the `phase function' $\Phi(t)$ satisfies the equation,
\begin{equation}\label{critical}
z\frac{\lambda'(z)}{\lambda(z)} -\eta=0 \qquad (z=z(c)). 
\end{equation}
Our strategy is to change the path $C$ so that it passes through 
a point satisfying the equation $\eqref{critical}$ and is to use Theorem 7.7.5 in \cite{Ho}.
To make it possible, we need to cut the complex plane so as to be able to take the eigenvalue $\lambda(z)$ depending holomorphically on $z$.

In what follows, we divide the discussion into the
following three cases:
\begin{equation}\label{division}
{\rm (A)}\hspace{10pt} \frac{1}{\sqrt{1+|b|^{2}}} <|\xi|<1,
\hspace{20pt}
{\rm (B)}\hspace{10pt} |a|<|\xi|<\frac{1}{\sqrt{1+|b|^{2}}},
\hspace{20pt}
{\rm (C)}\hspace{10pt} |\xi|=\frac{1}{\sqrt{1+|b|^{2}}}.
\end{equation}
In the rest of this section, the following two functions play important roles:
\begin{equation}\label{DandR}
D(\eta):=\frac{|b|+\sqrt{\eta^{2}-|a|^{2}}}{\sqrt{1-\eta^{2}}},\quad
r(\eta):=\frac{|b|\eta +\sqrt{\eta^{2}-|a|^{2}}}{|a|\sqrt{1-\eta^{2}}}\quad (|a|<|\eta|<1).
\end{equation}
Note that one has
\[
D(\eta)^{-1}=\frac{|b|-\sqrt{\eta^{2}-|a|^{2}}}{\sqrt{1-\eta^{2}}},\quad
r(\eta)^{-1}=\frac{|b|\eta -\sqrt{\eta^{2}-|a|^{2}}}{|a|\sqrt{1-\eta^{2}}}.
\]
The function $H_{Q}(\eta)$ defined in $\eqref{rateQ}$ can be expressed as
\begin{equation}\label{rateQ2}
H_{Q}(\eta)/2=\eta \log |r(\eta)| -\log D(\eta)=|\eta|\log r(|\eta|) -\log D(\eta).
\end{equation}

\subsection{Proof of Proposition ${\bf \ref{amp3hidden}}$ for the region (A)}
In this subsection, we assume that the parameter $\xi$ in the statement of Proposition $\ref{amp3hidden}$ is in the region (A).
\subsubsection{Some lemmas}
\label{lemmasA}
Let $\psi_{A}(\zeta)$ be a function in $\zeta \in \mb{C} \setminus [-1,1]$ defined by
\[
\psi_{A}(\zeta)=\zeta+(\zeta-1)\sqrt{\frac{\zeta +1}{\zeta -1}},
\]
where the branch of the square root is chosen so that it is positive on the positive real axis.
The function $\phi(z)=(z+z^{-1})/2$ maps the domain $\{|z|>1\}$ conformally and injectively onto $\mb{C} \setminus [-1,1]$, and
the function $\psi_{A}(\zeta)$ is its inverse. Thus, if we set
\[
\lambda_{A}(z):=\psi_{A}(|a|\phi(z)),\quad z \in D_{A},
\]
then the eigenvalues of $A(\omega z)$ are given by $\lambda_{A}(z)$ and $\lambda_{A}(z)^{-1}$,
where the domain $D_{A}$ is defined by
\[
D_{A}:=\mb{C} \setminus
\left(
S^{1} \cup \{0\} \cup [-(1+|b|)/|a|,-(1-|b|)/|a|] \cup [(1-|b|)/|a|,(1+|b|)/|a|]\right).
\]
The function $\lambda_{A}(z)$ is obviously holomorphic on the domain $D_{A}$.
Define the vectors $u_{A}(z),v_{A}(z) \in \mb{C}^{2}$ by
\begin{equation}\label{eigenVC}
u_{A}(z)=
\begin{pmatrix}
\frac{ab}{|a|}z \\
\lambda_{A}(z)-|a|z^{-1}
\end{pmatrix},\quad
v_{A}(z)=
\begin{pmatrix}
\frac{ab}{|a|}z \\
\lambda_{A}(z)^{-1}-|a|z^{-1}
\end{pmatrix}.
\end{equation}
One can easily check that the vector $u_{A}(z)$ ({\it resp}. $v_{A}(z)$) is an eigenvector of $A(\omega z)$ with the eigenvalue $\lambda_{A}(z)$ ({\it resp}. $\lambda_{A}(z)^{-1}$).
For any $\varphi=(\varphi_{1},\varphi_{2}) \in \mb{C}^{2}$, we set
\begin{equation}\label{coeff}
\varphi_{1A}(z)=\frac{|a|(|a|z^{-1}-\lambda_{A}(z)^{-1})\varphi_{1}+abz \varphi_{2}}{abz (\lambda_{A}(z)-\lambda_{A}(z)^{-1})},\quad
\varphi_{2A}(z)=-\frac{|a|(|a|z^{-1}-\lambda_{A}(z))\varphi_{1}+abz \varphi_{2}}{abz (\lambda_{A}(z)-\lambda_{A}(z)^{-1})}
\end{equation}
so that
\begin{equation}\label{EDeco}
\varphi=\varphi_{1A}(z)u_{A}(z) +\varphi_{2A}(z)v_{A}(z),\quad z \in D_{A}.
\end{equation}
For $\psi \in \mb{C}^{2}$, we set
\begin{equation}\label{fgApsi}
f_{A\psi}(z)=\varphi_{1A}(z)\ispa{u_{A}(z),\psi},\quad g_{A\psi}(z)=\varphi_{2A}\ispa{v_{A}(z),\psi}.
\end{equation}
Then the integral formula $\eqref{int00}$ is written as
\begin{equation}\label{int41}
\begin{gathered}
\ispa{U^{n}(\delta_{0} \otimes \varphi), \delta_{y} \otimes \psi}_{\ell}  =\omega^{-y}(I_{A1}(n)+I_{A2}(n)), \\
I_{A1}(n)=\frac{1}{2\pi i}\int_{C}z^{-y-1}\lambda_{A}(z)^{n}f_{A\psi}(z)\,dz,\quad
I_{A2}(n)=\frac{1}{2\pi i}\int_{C}z^{-y-1}\lambda_{A}(z)^{-n}g_{A\psi}(z)\,dz.
\end{gathered}
\end{equation}

\begin{lem}\label{cp}
Let $\eta \in \mb{R}$ with $|a|<|\eta|<1$. Then the solutions of the equation $\eqref{critical}$ with $\lambda(z)=\lambda_{A}(z)$, $z \in D_{A}$, are
given by $z=\pm ir(\eta)$, where $r(\eta)$ is defined in $\eqref{DandR}$.
\end{lem}

\begin{proof}
Since $\psi_{A}'(\zeta)=\frac{\psi_{A}(\zeta)}{\zeta -1}\sqrt{(\zeta-1)/(\zeta+1)}$, we have
\begin{equation}\label{cpe}
z \frac{\lambda_{A}'(z)}{\lambda_{A}(z)}=\frac{|a|\mu(z)}{|a|\phi(z)-1}\sqrt{\frac{|a|\phi(z)-1}{|a|\phi(z)+1}},\quad
\mu(z)=\frac{z-z^{-1}}{2}.
\end{equation}
From this and $\phi(z)^{2}-\mu(z)^{2}=1$, we see that the solution $z \in D_{A}$ of the equation $\eqref{critical}$ must satisfy
$\phi(z)^{2}=-(\eta^{2}-|a|^{2})/|a|^{2}(1-\eta^{2})$.
Hence $z$ is of the form $z=iy$ with a real number $y$. Since $\phi(iy)=i\mu(y)$, we see $\mu(y)=\pm \sqrt{\eta^{2}-|a|^{2}}/|a|\sqrt{1-\eta^{2}}$.
Since
\[
\sqrt{\frac{ir-1}{ir+1}}=
\begin{cases}
\frac{r+i}{\sqrt{r^{2}+1}} & (r>0),\\
-\frac{r+i}{\sqrt{r^{2}+1}} & (r<0),
\end{cases}
\]
we have $\phi(y)=\pm |b|\eta/|a|\sqrt{1-\eta^{2}}$. Our claim is a consequence of the relation $y=\mu(y)+\phi(y)$.
\end{proof}

The following lemma is used to find a suitable contour $C$ in $\eqref{int41}$.

\begin{lem}\label{ellipse1}
Let $r>(1+|b|)/|a|$ or $0<r<(1-|b|)/|a|$. Then we have
\[
|\lambda_{A}(\pm r)| \leq |\lambda_{A}(re^{it})| \leq |\lambda_{A}(\pm ir)| \quad (-\pi \leq t \leq \pi).
\]
The equality in the first inequality holds if and only if $t=0, \pm \pi$, and the equality in the second
holds if and only if $t=\pm \pi/2$.
\end{lem}

\begin{proof}
Let $r>(1+|b|)/|a|$, and let $C_{r}$ be the circle of radius $r$ centered at the origin. The image of $C_{r}$ under the map $\phi(z)=(z+z^{-1})/2$ is
the following ellipse
\[
\phi(C_{r}):\phi(re^{it})=\phi(r)\cos t +i\mu(r) \sin t \quad
(-\pi \leq t \leq \pi),
\]
and thus the curve $|a|\phi(C_{r})$ is also an ellipse, where, as before, we set $\mu(z)=(z-z^{-1})/2$. 
Take two positive numbers $x$, $y$ such that
\[
\phi(x)=|a|\phi(r),\quad \mu(y)=|a|\mu(r),\quad x,y>1.
\]
The numbers $x$, $y$ satisfying the above are uniquely determined.
The ellipse $|a|\phi(C_{r})$ lies outside $\phi(C_{x})$ and inside $\phi(C_{y})$. 
Furthermore $|a|\phi(C_{r})$ is tangent to $\phi(C_{x})$
at the points $\pm |a|\phi(r)=\pm \phi(x)$, and tangent to $\phi(C_{y})$
at the points $\pm i|a|\mu(r)=\pm i\mu(y)$. 
Hence the curve $\lambda_{A}(C_{r})=\psi_{A}(|a|\phi(C_{r}))$ lies between the two circles $C_{x}$, $C_{y}$.
Since the map $\psi_{A}$ is conformal, the curve $\lambda_{A}(C_{r})$ is tangent to $C_{x}$ at $\pm x$ from
the outside of $C_{x}$ and is tangent to $C_{y}$ at $\pm iy$ from the inside of $C_{y}$.
Thus we have $x \leq |\lambda_{A}(re^{it})| \leq y$.
Since $\phi(iy)=i\mu(y)$, we see
\[
\pm i y=\psi_{A}(\phi(\pm iy))=\psi_{A}(\pm i \mu(y))=\psi_{A}(\pm i|a|\mu(r))=\psi_{A}(|a|\phi(\pm ir))=\lambda_{A}(\pm i r),
\]
and similarly $\pm x =\lambda_{A}(\pm r)$. This proves the desired inequalities. 
The conditions for the equalities are a consequence of the fact that the curve $\lambda_{A}(C_{r})$ is tangent to $C_{x}$ only at $\pm x$ and to $C_{y}$ only at $\pm iy$.
Next, let us consider the case where $0<r<(1-|b|)/|a|$. Since $|a|/(1-|b|)=(1+|b|)/|a|$, we have $(1+|b|)/|a|<r^{-1}$.
Thus, from what we have just proved for the case where $r>(1+|b|)/|a|$, we see
\[
|\lambda_{A}(\pm r^{-1})| \leq |\lambda_{A}(r^{-1}e^{it})| \leq |\lambda_{A}(\pm ir^{-1})|.
\]
Noting $\phi(z^{-1})=\phi(z)$, we see $\lambda_{A}(z^{-1})=\lambda_{A}(z)$, and hence replacing $t$ by $-t$ in the
above inequality, we obtain the desired inequalities together with the conditions for the equalities.
\end{proof}

\subsubsection{Proof of $\eqref{amp3AF}$ for the case $\xi>0$ in the region {\rm (A)}}
\label{positiveA}
The condition $1/\sqrt{1+|b|^{2}}<\xi<1$ is equivalent to the inequality $(1+|b|)/|a|<r(\xi)$.
Set $\xi_{n}=y_{n}/n$ as in the statement of Proposition $\ref{amp3hidden}$ and
take the circle $C_{n}=C_{r(\xi_{n})}$ of radius $r(\xi_{n})$ for the contour $C$ in $\eqref{int41}$.
Since $\xi_{n} \to \xi$ as $n \to \infty$, we have $r(\xi_{n})>(1+|b|)/|a|$ for every sufficiently large $n$.
The integral $I_{A1}(n)$ with $y=y_{n}$ defined in $\eqref{int41}$ can be written in the form
\[
\begin{split}
I_{A1}(n) & =\frac{e^{-n[\xi_{n}\log r(\xi_{n})-\log D(\xi_{n})]}}{2\pi}
\int_{-\pi}^{\pi}e^{-iy_{n}t}
\left[
\frac{\lambda_{A}(r(\xi_{n})e^{it})}{D(\xi_{n})}
\right]^{n}
f_{A\psi}(r(\xi_{n})e^{it})\,dt \\
& = \frac{e^{-nH_{Q}(\xi_{n})/2}}{2\pi}
\int_{-\pi}^{\pi}e^{-iy_{n}t}
\left[
\frac{\lambda_{A}(r(\xi_{n})e^{it})}{D(\xi_{n})}
\right]^{n}
f_{A\psi}(r(\xi_{n})e^{it})\,dt,
\end{split}
\]
where $D(\xi)$ is defined in $\eqref{DandR}$ and we have used $\eqref{rateQ2}$.
A direct computation shows
\begin{equation}\label{DandE}
\lambda_{A}(\pm ir(\eta))=\pm i D(\eta),\quad |a| <|\eta| <1.
\end{equation}
By Lemma $\ref{ellipse1}$, for each $\eta$ with $\frac{1}{\sqrt{1+|b|^{2}}} <\eta <1$,
we have the inequality $|\lambda_{A}(r(\eta)e^{it})/D(\eta)| \leq 1$,
and the equality holds if and only if $t=\pm \pi/2$ on the interval $[-\pi,\pi]$.
Take a function $\chi(t) \in C_{0}^{\infty}(-\pi/4,\pi/4)$ with $\chi(t)=1$ near $t=0$. 
Setting $\chi_{1}(t)=1-\chi(t-\pi/2)-\chi(t+\pi/2)$, the integral $I_{A1}(n)$ can be written as
\[
I_{A1}(n)=e^{-nH_{Q}(\xi_{n})/2} (A_{+}(n)+A_{-}(n)+R(n)),
\]
where we set
\[
\begin{gathered}
A_{\pm}(n) =\frac{e^{\mp i\pi y_{n}/2}}{2\pi}\int e^{-iy_{n}t}f_{A\psi}(\pm ir(\xi_{n})e^{it})\chi(t)
\left[
\frac{\lambda_{A}(\pm i r(\xi_{n})e^{it})}{D(\xi_{n})}
\right]^{n}\,dt, \\
R(n)=\frac{1}{2\pi}\int e^{-iy_{n}t}f_{A\psi}(r(\xi_{n})e^{it})\chi_{1}(t)
\left[
\frac{\lambda_{A}(r(\xi_{n})e^{it})}{D(\xi_{n})}
\right]^{n}\,dt.
\end{gathered}
\]
We can choose a positive constant $0<c<1$ such that
\[
\left|
\frac{\lambda_{A}(r(\xi_{n})e^{it})}{D(\xi_{n})}
\right| \leq c
\]
for every $n$ and every $t$ near the support of $\chi_{1}(t)$.
Setting $R=-\log c>0$, we see $|R(n)| =O(e^{-nR})$.
Next, take a look at the integrals $A_{\pm}(n)$. By $\eqref{DandE}$,
one can take $\chi(t)$ having a support so small that a branch of the logarithm $\log (\lambda_{A}( \pm ir(\xi_{n})e^{it}))$
exists near the support of $\chi(t)$.
Then the integrals $A_{\pm}(n)$ are written as
\[
A_{\pm}(n)=\frac{e^{\mp i\pi y_{n}/2}}{2\pi}
\int e^{n\Phi_{\pm}(\xi_{n},t)}f_{A\psi}(\pm i r(\xi_{n})e^{it})\chi(t)\,dt,
\]
where, for $\eta$ close to the fixed parameter $\xi$, the function $\Phi_{\pm}(\eta,t)$ is defined by
\[
\Phi_{\pm}(\eta,t)=\log (\lambda_{A}(\pm ir(\eta)e^{it})) -\log D(\eta)-i\eta t.
\]
We have $\Phi_{\pm}(\eta,0)=\pm i\pi/2$. The real part of $\Phi_{\pm}(\eta,t)$ is non-positive, and is zero if and only if $t=0$
near the support of $\chi(t)$. If we put $z(t)=\pm ir(\eta)e^{it}$, then the first and the second derivatives are given by
\[
\begin{split}
\partial_{t}\Phi_{\pm}(\eta,t) & =
i\left[
z(t) \frac{\lambda_{A}'(z(t))}{\lambda_{A}(z(t))}-\eta
\right],\\
\partial_{t}^{2}\Phi_{\pm}(\eta,t) & =
-z(t)
\left[
\frac{\lambda_{A}'(z(t))}{\lambda_{A}(z(t))}
+z(t)\frac{\lambda_{A}''(z(t))\lambda_{A}(z(t))-\lambda_{A}'(z(t))^{2}}{\lambda_{A}(z(t))^{2}}
\right].
\end{split}
\]
This shows that, near the support of $\chi(t)$, $\partial_{t}\Phi_{\pm}(\eta,t)=0$ if and only if $t=0$.
A direct computation using $\eqref{cpe}$ yields
\[
\frac{d}{dz}\left(
z\frac{\lambda_{A}'}{\lambda_{A}}
\right)
=
-\frac{|a||b|^{2}\phi(z)}{z(|a|\phi(z)-1)(|a|\phi(z)+1)^{2}}
\sqrt{\frac{|a|\phi(z)+1}{|a|\phi(z)-1}}.
\]
Substituting $z=\pm ir(\eta)$ for this identity, we get
\[
\partial_{t}^{2}\Phi_{\pm}(\eta,0)=-\frac{(1-\eta^{2})\sqrt{\eta^{2}-|a|^{2}}}{|b|},
\]
and hence the critical point $t=0$ is non-degenerate.
Note that each higher derivative of $\Phi_{\pm}(\eta,t)$ is bounded and
$|\partial_{t}^{2}\Phi_{\pm}(\eta,t)|$ is bounded from below when $(\eta,t)$ is close to $(\xi,0)$.
This allows us to use Theorem 7.7.5 in \cite{Ho} to obtain
\[
A_{\pm}(n)=\frac{e^{\pm i\pi (n-y_{n})/2}}{\sqrt{2\pi n}}
\sqrt{
\frac{|b|}{(1-\xi_{n}^{2})\sqrt{\xi_{n}^{2}-|a|^{2}}}
}
\left(
f_{A\psi}(\pm ir(\xi_{n}))+O(1/n)
\right).
\]
The leading term in the above is a function in $\xi_{n}$. Recall our assumption $\xi_{n}=\xi+O(1/n)$.
Hence the functions in $\xi_{n}$ in the leading term, except the exponential term,
can be replaced by the values at $\xi$ of the same functions plus terms of order $O(1/n)$.
A simple computation shows
\begin{equation}\label{Fpsi1}
f_{A\psi}(-ir)=f_{A\psi}(ir),\quad (r \in \mb{R} \setminus \{0,\pm 1\}).
\end{equation}
Since $e^{-i\pi (n-y_{n})/2}=(-1)^{n-y_{n}}e^{i\pi (n-y_{n})/2}$, we get
\begin{equation}\label{asy11}
\begin{split}
I_{A1}(n) & =\frac{e^{-nH_{Q}(\xi_{n})/2}}{\sqrt{2\pi n}}
\sqrt{\frac{|b|}{(1-\xi^{2})\sqrt{\xi^{2}-|a|^{2}}}}
e^{i\pi (n-y_{n})/2} \\
& \hspace{30pt} \times  \left(
(1+(-1)^{n+y_{n}})f_{A\psi}(ir(\xi))+O(1/n)
\right).
\end{split}
\end{equation}
Next, let us consider the integral $I_{A2}(n)$ in $\eqref{int41}$. Lemma $\ref{ellipse1}$ shows
\begin{equation}\label{ftE}
|\lambda_{A}(r(\eta)e^{it})^{-1}| \leq E(\eta):=|\lambda_{A}(\pm r(\eta))^{-1}|=\frac{|b|\eta -\sqrt{(1+|b|^{2})\eta^{2}-1}}{\sqrt{1-\eta^{2}}}
\end{equation}
for $1/\sqrt{1+|b|^{2}}<\eta<1$. From this, we have
\[
|I_{A2}(n)| \leq e^{-n(\xi_{n}\log r(\xi_{n}) -\log E(\xi_{n}))}\max_{|z|=r(\xi_{n})}|g_{A\psi}(z)|,
\]
and hence, by $\eqref{rateQ2}$,
\[
e^{nH_{Q}(\xi_{n})/2}|I_{A2}(n)| \leq C e^{-n(\log D(\xi_{n})-\log E(\xi_{n}))}.
\]
It is easy to see that $D(\eta)/E(\eta) > 1$ for $1/\sqrt{1+|b|^{2}}<\eta<1$. Since $\xi_{n}=\xi+O(1/n)$, one can find
a positive constant $c$ such that $D(\xi_{n})/E(\xi_{n}) \geq 1+c$ uniformly in $n$.
This leads us to 
\begin{equation}\label{asy12}
I_{A2}(n)=\frac{e^{-nH_{Q}(\xi_{n})/2}}{\sqrt{2\pi n}}O(n^{1/2}e^{-n\log (1+c)}).
\end{equation}
Recalling that $\ispa{U^{n}(\delta_{0} \otimes \varphi), \delta_{y} \otimes \psi}_{\ell}=0$ when $n+y$ is odd, and in view of 
$\eqref{asy11}$ and $\eqref{asy12}$, we get the formula $\eqref{amp3AF}$ with
\begin{equation}\label{FApsi1}
F_{\psi}(\xi)=f_{A\psi}(ir(\xi)).
\end{equation}
\subsubsection{Proof of $\eqref{amp3AF}$ for the case $\xi<0$ in the region {\rm (A)}}
\label{negativeA}
In this case, we have $-1<\xi_{n}<-1/\sqrt{1+|b|^{2}}$, $-(1-|b|)/|a| <r(\xi_{n})<0$ for every sufficiently large $n$.
We take the circle of radius $-r(\xi_{n})=|r(\xi_{n})|$ for the contour $C$ in $\eqref{int41}$
and discuss in a way similar as in the case where $\xi>0$.
Note that, when $-1<\eta<-1/\sqrt{1+|b|^{2}}$, we have $|r(\eta)|=-r(\eta)=r(|\eta|)^{-1}$.
Since $\lambda_{A}(z^{-1})=\lambda_{A}(z)$, the integral $I_{A1}(n)$ with $y=y_{n}$ in $\eqref{int41}$ is expressed as
\[
I_{A1}(n)=\frac{e^{-nH_{Q}(\xi_{n})/2}}{2\pi}
\int_{-\pi}^{\pi}e^{-iy_{n}t}
\left[
\frac{\lambda_{A}(r(|\xi_{n}|)e^{-it})}{D(|\xi_{n}|)}
\right]^{n}f_{A\psi}(-r(\xi_{n})e^{it})\,dt.
\]
By Lemma $\ref{ellipse1}$, we have $|\lambda_{A}(r(|\xi_{n}|)e^{-it})| \leq D(|\xi_{n}|)$. Since $D(\eta)=D(|\eta|)$,
the same discussion as in the case where $\xi>0$ leads us to the formula $\eqref{asy11}$ by replacing $e^{i\pi(n-y_{n})/2}$ with $e^{i\pi(n-|y_{n}|)/2}$.
For the integral $I_{A2}(n)$, remark that $\lambda_{A}(r(\eta))=-\lambda_{A}(r(|\eta|))$ for $-1<\eta<-1/\sqrt{1+|b|^{2}}$, and hence, by Lemma $\ref{ellipse1}$,
\[
|\lambda_{A}(-r(\eta)e^{it})^{-1}|=|\lambda_{A}(r(|\eta|)e^{-it})^{-1}| \leq |\lambda_{A}(r(|\eta|))^{-1}|=E(|\eta|).
\]
From this, it is straightforward to obtain
\[
e^{nH_{Q}(\xi_{n})/2}|I_{A2}(n)| \leq Ce^{-n(\log D(|\xi_{n}|)-\log E(|\xi_{n}|))},
\]
which yields the same estimate as $\eqref{asy12}$.
Therefore the formula $\eqref{amp3AF}$ with the function $F_{\psi}(\xi)$ given by $\eqref{FApsi1}$ is proved.

\subsection{Proof of Proposition ${\bf \ref{amp3hidden}}$ for the region (B)}
In this subsection, we assume that the parameter $\xi$ is in the region (B).

\subsubsection{Some lemmas}
\label{lemmasB}
Define the function $\psi_{B}(\zeta)$ by
\[
\psi_{B}(\zeta)=\zeta +i\sqrt{1-\zeta^{2}},\quad \zeta \in \mb{C} \setminus \{x \in \mb{R}\,;\,|x| \geq 1\},
\]
where the branch of the square root is chosen to be positive for positive reals.
Then $\psi_{B}$ is holomorphic in the domain indicated as above, and $\psi_{B}(\zeta)^{-1}=\zeta-i\sqrt{1-\zeta^{2}}$.
The function $\psi_{B}$ maps the above domain injectively and conformally onto the upper half plane,
and it is the inverse of the map $\phi(z)=(z+z^{-1})/2$ restricted
to the upper half plane. We set, as in $\eqref{eigen1}$,
\[
\lambda_{B}(z):=\psi_{B}(|a|\phi(z)),
\]
which is holomorphic on the domain $D_{B}=D$ given in $\eqref{Dom}$.
Then the eigenvalues of $A(\omega z)$ for $z \in D_{B}$ are $\lambda_{B}(z)$, $\lambda_{B}(z)^{-1}$.
Since $\psi_{B}(-\zeta)=-\psi_{B}(\zeta)^{-1}$, we have $\lambda_{B}(-z)=-\lambda_{B}(z)^{-1}$.
Note that, in this case, $\im \lambda_{B}>0$.
For $z \in D_{B}$, define the vectors $u_{B}(z), v_{B}(z) \in \mb{C}^{2}$ by the formula $\eqref{eigenVC}$ with replacing $\lambda_{A}$ with $\lambda_{B}$.
Similarly, define the functions $\varphi_{1B}(z)$, $\varphi_{2B}(z)$ on $D_{B}$ by the formula $\eqref{coeff}$.
Then $u_{B}(z)$ ({\it resp}. $v_{B}(z)$) is an eigenvector of $A(\omega z)$ with the eigenvalue $\lambda_{B}(z)$ ({\it resp}. $\lambda_{B}(z)^{-1}$),
and, for any $\varphi \in \mb{C}^{2}$, the formula $\eqref{EDeco}$ holds by replacing the letter `$A$' with `$B$'.
Define, for $\psi \in \mb{C}^{2}$, the functions $f_{B\psi}$, $g_{B\psi}$ by $\eqref{fgApsi}$ by replacing `$A$' with `$B$'.
The integral formula corresponding to $\eqref{int41}$ which is obtained by replacing `$A$'  with `$B$' takes the form
\begin{equation}\label{int41B}
\begin{gathered}
\ispa{U^{n}(\delta_{0} \otimes \varphi), \delta_{y} \otimes \psi}_{\ell}  =\omega^{-y}(I_{B1}(n)+I_{B2}(n)), \\
I_{B1}(n)=\frac{1}{2\pi i}\int_{C}z^{-y-1}\lambda_{B}(z)^{n}f_{B\psi}(z)\,dz,\quad
I_{B2}(n)=\frac{1}{2\pi i}\int_{C}z^{-y-1}\lambda_{B}(z)^{-n}g_{B\psi}(z)\,dz.
\end{gathered}
\end{equation}
It is easy to see that $u_{B}(-z)=-v_{B}(z)$,
$\varphi_{1B}(-z)=-\varphi_{2B}(z)$, $f_{B\psi}(-z)=g_{B\psi}(z)$, and hence we have
\[
I_{B2}(n)=(-1)^{n+y}I_{B1}(n),\quad
\ispa{U^{n}(\delta_{0} \otimes \varphi), \delta_{y} \otimes \psi}_{\ell}=\omega^{-y}(1+(-1)^{n+y})I_{B1}(n).
\]
Therefore, to prove Proposition $\ref{amp3hidden}$, it is enough to consider the
asymptotic behavior of the integral $I_{B1}(n)$.
The following lemma is used to specify the solutions of the equation $\eqref{critical}$ for $\lambda=\lambda_{B}$.

\begin{lem}\label{cpB}
Let $\eta \in \mb{R}$ with $|a|<|\eta|<1$. Then the solutions of the equation $\eqref{critical}$ with $\lambda(z)=\lambda_{B}(z)$, $z \in D_{B}$, are
given by $z=ir(\eta)$ and $z=ir(\eta)^{-1}$, where $r(\eta)$ is defined in $\eqref{DandR}$.
\end{lem}

\begin{proof}
We have $\psi_{B}'(\zeta)=-i\psi_{B}(\zeta)/\sqrt{1-\zeta^{2}}$, and hence
\begin{equation}\label{derv1}
z\frac{\lambda_{B}'(z)}{\lambda_{B}(z)}=-i\frac{|a|\mu(z)}{\sqrt{1-|a|^{2}\phi(z)^{2}}},
\end{equation}
where, as before, $\mu(z)=(z-z^{-1})/2$.
Then a similar computation as in the proof of Lemma $\ref{cp}$ shows the claim.
\end{proof}

The following is used to find a suitable path $C$ in the integral formula $\eqref{int41B}$.

\begin{lem}\label{ellipseB}
Let $0<c<1$ and $\mu>0$. Let $\zeta=\zeta(t)$ $(-\pi \leq t \leq \pi)$ be an ellipse given by
\[
\zeta(t)=c \cos t +i\mu \sin t.
\]
Then we have
\[
|\psi_{B}(-i\mu)| \leq |\psi_{B}(\zeta(t))| \leq |\psi_{B}(i\mu)|.
\]
The equality in the first inequality holds if and only if $t=-\pi/2$, and the equality in the second holds if and only if $t=\pi/2$.
\end{lem}
\begin{proof}
For each $\zeta \in \mb{C} \setminus \{x \in \mb{R}\,;\,|x| \geq 1\}$, set
\begin{equation}\label{Falbe}
\alpha=\alpha(\zeta)=\re(\sqrt{1-\zeta^{2}}), \quad
\beta=\beta(\zeta)=\im (\sqrt{1-\zeta^{2}}).
\end{equation}
If $\zeta=x+iy$, one has $\alpha^{2}-\beta^{2}=1-x^{2}+y^{2}$, $\alpha \beta=-xy$.
Since $\alpha>0$, we see $\psi_{B}(\zeta)=(\alpha +y)(x/\alpha +i)$ and $\psi_{B}(\zeta)^{-1}=(\alpha -y)(x/\alpha -i)$.
From this, we have $(\alpha^{2}-y^{2})(1+x^{2}/\alpha^{2})=1$ which
shows that $\alpha>|y|$ and $\im \psi_{B}(\zeta)=\alpha +y>0$. It is straightforward to check
\begin{equation}\label{Fal}
2\alpha^{2}(\zeta)=(1-x^{2}+y^{2})+\sqrt{(1-x^{2}+y^{2})^{2}+4x^{2}y^{2}},\quad \zeta=x+iy.
\end{equation}
Now, put $\alpha(t)=\alpha(\zeta(t))$. 
From the facts mentioned above, it follows that $|\psi_{B}(\zeta(t))|^{2}=(\alpha(t)+\mu \sin t)/(\alpha(t)-\mu \sin t)$.
Setting
\[
A(t):=\frac{\alpha(t)}{\mu}\left(
\alpha(t) -\mu \sin^{2}t
\right)^{2}\frac{d}{dt}
|\psi(\zeta(t))|^{2},
\]
we see $A(t)=2\alpha(t)^{2}\cos t -2\alpha(t)\alpha'(t) \sin t$. 
A direct computation using $\eqref{Fal}$ shows
\[
A(t)=\frac{\cos t}{\sqrt{D(t)}}
\left[
(1-c^{2})\sqrt{D(t)} +(1-c^{2})(1-c^{2}\cos^{2}t)+\mu^{2}(1+c^{2})\sin^{2}t
\right],
\]
where we set
\[
D(t)=(1-c^{2}\cos^{2}t+\mu^{2}\sin^{2}t)^{2}+4c^{2}\mu^{2}\sin^{2}t\cos^{2}t.
\]
By the assumption $c<1$, $A(t)=0$ 
if and only if $\cos t=0$. Note that $A(t)=0$ is equivalent to say that $(d/dt)|\psi_{B}(\zeta(t))|^{2}=0$. On the other hand, we have
\[
|\psi_{B}(\zeta(-\pi/2))|=|\psi_{B}(-i\mu)|=-\mu+\sqrt{\mu^{2}+1}<|\psi_{B}(\zeta(\pi/2))| =|\psi(i\mu)|=\mu+\sqrt{\mu^{2}+1}.
\]
From this, our claim follows.
\end{proof}

\subsubsection{Proof of $\eqref{amp3AF}$ for the case $\xi>0$ in the region {\rm (B)}}
\label{positiveB}
In this case, we have $|a|<\xi<1/\sqrt{1+|b|^{2}}$ and $1<r(\xi)<(1+|b|)/|a|$.
We take the circle
\[
C=C_{n}: z(t)=z_{n}(t):=r(\xi_{n})e^{it} \quad (-\pi \leq t \leq \pi)
\]
of radius $r(\xi_{n})$ for the contour $C$ in the integrals $\eqref{int41B}$. Then the curve $|a|\phi(z(t))$ is an ellipse since
\[
|a|\phi(z(t))=|a|\phi(r(\xi_{n}))\cos t +i|a|\mu(r(\xi_{n})) \sin t.
\]
Since $1<r(\xi)<(1+|b|)/|a|$, we have $|a|<|a|\phi(r(\xi_{n}))<1$ and $0<|a|\mu(r(\xi_{n}))$.
By Lemma $\ref{ellipseB}$, we see
\[
|\psi_{B}(-i|a|\mu(r(\xi_{n})))| \leq |\lambda_{B}(z(t))|=|\psi_{B}(|a|\phi(z_{n}(t)))| \leq |\psi_{B}(i|a|\mu(r(\xi_{n})))|.
\]
Note that $\psi_{B}(\pm i|a|\mu(r(\xi_{n})))=\lambda_{B}(\pm i r(\xi_{n}))$. For $|a|<\eta<1$, it is easy to check that
\begin{equation}\label{eigenBB}
\lambda_{B}(i r(\eta))=iD(\eta),\quad
\lambda_{B}(-ir(\eta))=-\lambda_{B}(ir(\eta))^{-1}=iD(\eta)^{-1}=i\frac{|b|-\sqrt{\eta^{2}-|a|^{2}}}{\sqrt{1-\eta^{2}}}.
\end{equation}
The integral $I_{B1}(n)$ in $\eqref{int41B}$ is written as
\[
I_{B1}(n)=\frac{e^{-nH_{Q}(\xi_{n})/2}}{2\pi}
\int_{-\pi}^{\pi}
e^{-iy_{n}t}
\left[
\frac{\lambda_{B}(z_{n}(t))}{D(\xi_{n})}
\right]^{n}
f_{B\psi}(z_{n}(t))\,dt.
\]
By Lemma $\ref{ellipseB}$, when $|a|<\eta<1/\sqrt{1+|b|^{2}}$,
the equality $|\lambda_{B}(r(\eta)e^{it})|=D(\eta)$ ($-\pi \leq t \leq \pi$) holds 
if and only if $t=\pi/2$. Therefore, for a compact neighborhood $K \subset (|a|,1/\sqrt{1+|b|^{2}})$ of $\xi$, there exists a constant $c>0$ such that
\[
\frac{|\lambda_{B}(r(\eta)e^{it})|}{D(\eta)} \leq e^{-c},\quad \eta \in K,\ \ -\pi/2 \leq t \leq \pi/2, \ \ |t-\pi/2| \geq \ve,
\]
where we fix $\ve$ with $0<\ve <\pi/8$. Take $\chi(t) \in C_{0}^{\infty}(\mb{R})$ such that $\chi(t)=1$ for $|t| \leq \ve$ and $\chi(t)=0$
for $|t| \geq 2\ve$. By using the above estimate, the integral $I_{B1}(n)$ can be written as
\[
\begin{split}
I_{B1}(n) & =\frac{e^{-nH_{Q}(\xi_{n})/2}}{2\pi}  (J(n)+O(e^{-cn})),\\
J(n) & = \int e^{-iy_{n}t}
\left[
\frac{\lambda_{B}(r(\xi_{n})e^{it})}{D(\xi_{n})}
\right]^{n}
\chi(t-\pi/2)f_{B\psi}(r(\xi_{n})e^{it})\,dt\\
& =e^{-i\pi y_{n}/2}\int e^{-iy_{n}t}\left[
\frac{\lambda_{B}(ir(\xi_{n})e^{it})}{D(\xi_{n})}
\right]^{n}\chi(t)f_{B\psi}(ir(\xi_{n})e^{it})\,dt.
\end{split}
\]
By $\eqref{eigenBB}$, one can choose $\ve>0$ so that
a branch of the logarithm $\log \lambda_{B}(ir(\eta)e^{it})$ is well-defined when $|t| \leq 2\ve$ and $\eta$ stays in
a compact neighborhood $K$ of $\xi$. Define
\[
\Phi(\eta,t)=\log \lambda_{B}(ir(\eta)e^{it})-i\eta t -\log D(\eta),\quad \eta \in K,\ \ |t|<2\ve
\]
so that the integral $J(n)$ can be written as
\[
J(n)=e^{-i\pi y_{n}/2}\int e^{n\Phi(\xi_{n},t)}\chi(t)f_{B\psi}(ir(\xi_{n})e^{it})\,dt.
\]
Note that $\Phi(\eta,0)=i\pi/2$. By Lemma $\ref{ellipseB}$, we have $\re \Phi(\eta,t) \leq 0$.
We can also check
\[
\begin{split}
\partial_{t}\Phi(\eta,t) & = i
\left(
z\frac{\lambda_{B}'(z)}{\lambda_{B}(z)}-\eta
\right),\\
\partial_{t}^{2}\Phi(\eta,t) & = -z
\left(
\frac{\lambda_{B}'(z)}{\lambda_{B}(z)}
+z \frac{\lambda_{B}''(z)\lambda_{B}(z)-\lambda_{B}'(z)^{2}}{\lambda_{B}(z)^{2}}
\right) \quad (z=ir(\eta)e^{it}).
\end{split}
\]
Thus, by Lemma $\ref{cpB}$, $\partial_{t}\Phi(\eta,t)=0$ if and only if $t=0$.
Furthermore, since, by $\eqref{derv1}$,
\[
\frac{d}{dz}\left(z \frac{\lambda_{B}'(z)}{\lambda_{B}(z)}\right)
=\frac{\lambda_{B}'(z)}{\lambda_{B}(z)}+z\frac{\lambda_{B}''(z)\lambda_{B}(z)-\lambda_{B}'(z)^{2}}{\lambda_{B}(z)^{2}}
=-i\frac{|a||b|^{2}\phi(z)}{z(1-|a|^{2}\phi(z)^{2})^{3/2}},
\]
we have
\[
\partial_{t}^{2}\Phi(\eta,0)=-\frac{|a||b|^{2}\mu(r(\eta))}{(1+|a|^{2}\mu(r(\eta))^{2})^{3/2}}=-\frac{(1-\eta^{2})\sqrt{\eta^{2}-|a|^{2}}}{|b|}.
\]
Therefore, the critical point $t=0$ of $\Phi(\eta,t)$ is non-degenerate, and hence Theorem 7.7,5 in \cite{Ho} gives
\[
J(n)=\sqrt{
\frac{2\pi}{n}
\frac{|b|}{(1-\xi_{n}^{2})\sqrt{\xi_{n}^{2}-|a|^{2}}}
}
e^{i\pi (n-y_{n})/2}
\left(
f_{B\psi}(ir(\xi_{n}))+O(1/n)
\right).
\]
Since $\xi_{n}=\xi +O(1/n)$, one can replace $\xi_{n}$ with $\xi$ in the above expression to obtain
\[
I_{B1}(n)=\frac{e^{-nH_{Q}(\xi_{n})/2}}{\sqrt{2\pi n}}e^{i\pi (n-y_{n})/2}
\sqrt{
\frac{|b|}{(1-\xi^{2})\sqrt{\xi^{2}-|a|^{2}}}
}
\left(
f_{B\psi}(ir(\xi))+O(1/n)
\right).
\]
This shows the formula $\eqref{amp3AF}$ with
\begin{equation}\label{FBpsi1}
F_{\psi}(\xi)=f_{B\psi}(ir(\xi)).
\end{equation}

\subsubsection{Proof of $\eqref{amp3AF}$ for the case $\xi<0$ in the region {\rm (B)}}
\label{negativeB}
In this case, we have $-1<r(\xi_{n})<-(1-|b|)/|a|$ for every sufficiently large $n$.
We take the circle of radius $|r(\xi_{n})|=r(|\xi_{n}|)^{-1}$ for the contour $C$ in the integral $I_{B1}(n)$ in $\eqref{int41B}$. 
Then the argument can be carried out in the same way as in the case where $\xi>0$. 
Note that $\lambda_{B}(|r(\xi_{n})|e^{it})=\lambda_{B}(r(|\xi_{n}|)e^{-it})$. By Lemma $\ref{ellipseB}$,
the integral $I_{B1}(n)$ is written as
\[
\begin{split}
I_{B1}(n) & =\frac{e^{-nH_{Q}(\xi_{n})/2}}{2\pi}
\left(
J(n)+O(e^{-cn})
\right), \\
J(n) & =  \int e^{-iy_{n}t}
\left[
\frac{\lambda_{B}(r(|\xi_{n}|)e^{-it})}{D(|\xi_{n}|)}
\right]^{n}
\chi(t+\pi/2)
f_{B\psi}(r(|\xi_{n}|)^{-1}e^{it})\,dt \\
& = e^{i\pi y_{n}/2}\int e^{n\Phi(\xi_{n},t)}\chi(t)f_{B\psi}(-ir(|\xi_{n}|)^{-1}e^{it})\,dt,
\end{split}
\]
where $\chi(t) \in C_{0}^{\infty}(\mb{R})$ is chosen in a similar way as in the proof for the case where $\xi>0$, 
and the function $\Phi(\eta,t)$ is defined by 
\[
\Phi(\eta,t)=\log \lambda_{B}(ir(|\eta|)e^{-it})-\log D(|\eta|) -i\eta t.
\]
Note that $\Phi(\eta,t)$ is defined on a neighborhood of $(\xi,0)$, and $\Phi(\eta,0)=i\pi/2$.
By Lemma $\ref{ellipseB}$, we see $\re\Phi(\eta,t) \leq 0$.
A direct computation leads us to
\[
\begin{split}
\partial_{t}\Phi(\eta,t) & = -i\left(
z\frac{\lambda_{B}'(z)}{\lambda_{B}(z)}+\eta
\right),\\
\partial_{t}^{2}\Phi(\eta,t) & = -z
\left(
\frac{\lambda_{B}'(z)}{\lambda_{B}(z)}+z\frac{\lambda_{B}''(z)\lambda_{B}(z)-\lambda_{B}'(z)^{2}}{\lambda_{B}(z)^{2}}
\right)\quad (z=ir(|\eta|)e^{-it}).
\end{split}
\]
From this, it follows that $\partial_{t}\Phi(\eta,t)=0$ if and only if $t=0$. This critical point $t=0$ is non-degenerate since
\[
\partial_{t}^{2}\Phi(\eta,0)=-\frac{(1-\eta^{2})\sqrt{\eta^{2}-|a|^{2}}}{|b|}.
\]
Thus an application of Theorem 7.7.5 in \cite{Ho} gives 
\[
J(n)=\sqrt{
\frac{2\pi}{n}\frac{|b|}{(1-\xi^{2})\sqrt{\xi^{2}-|a|^{2}}}
}e^{i(n+y_{n})\pi/2}
\left(
f_{B\psi}(-ir(|\xi|)^{-1})+O(1/n)
\right).
\]
Noting $-r(|\xi|)^{-1}=r(\xi)$, we see that the asymptotic formula $\eqref{amp3AF}$ holds with $F_{\psi}(\xi)$ in $\eqref{FBpsi1}$ whenever $-1/\sqrt{1+|b|^{2}}<\xi<-|a|$.

\subsection{Proof of Proposition ${\bf \ref{amp3hidden}}$ for the case (C)}
In this case, we have
\[
\xi=\xi_{\pm}:=\pm 1/\sqrt{1+|b|^{2}}, \quad r(\xi_{\pm})=\pm r_{\pm},\quad r_{\pm}:=(1 \pm |b|)/|a|.
\]
We use the function $\lambda_{B}(z)$ as in the case of the region (B).
But, since $z=(1 \pm |b|)/|a|$ are branched points of the function $\lambda_{B}(z)$,
we need to choose a contour $C$ in $\eqref{int41B}$ so that it does not pass through these points.
This is done by using Lemma $\ref{ellipseC}$ below. 
Let $\ve>0$ be a small constant satisfying
\begin{equation}\label{veRG1}
0<\ve<\min
\left\{
1,\ r_{+}^{2},\ r_{-}^{2},\
\left(-3+\sqrt{5+4r_{+}^{2}}\right)/2,\
\left(3-\sqrt{5+4r_{-}^{2}}\right)/2
\right\}.
\end{equation}
Consider the following function: 
\[
z_{\pm}(t)=z_{\pm}(r,s,t):=(r_{\pm}^{2}+s)^{1/2}\cos t +i(r_{\pm}^{2}+r)^{1/2}\sin t,
\]
where the parameters, $r, s, t$, satisfy
\begin{equation}\label{PR}
|r| \leq \ve,\quad |t| \leq \pi,\quad
\begin{cases}
-\ve  \leq s <0 & \mbox{for $z_{+}(t)$},\\
0< s \leq \ve & \mbox{for $z_{-}(t)$}.
\end{cases}
\end{equation}
Note that if $\ve>0$ satisfies $\eqref{veRG1}$, we have
\[
0<\ve<1,\quad \ve<r_{\pm}^{2},\quad \sqrt{r_{-}^{2}+\ve}<1-\ve<1+\ve<\sqrt{r_{+}^{2}-\ve},
\]
and the curve $z_{\pm}(t)$ is an ellipse in the domain $D_{B}$.

\begin{lem}\label{ellipseC}
Let $K \subset \mb{R}$ be a bounded closed interval. Then there exists $\ve>0$ satisfying $\eqref{veRG1}$ such that
\begin{enumerate}
\item for every $\xi \in K$, $s \in (-\ve,0)$ and $|r|<\ve$ satisfying $|r-s| \leq 2|s|$,
the function
$$
\log |\lambda_{B}(z_{+}(r,s,t))|-\xi\log |z_{+}(r,s,t)|
$$
has critical points only at $t=\pm \pi/2$, and
\item for every $\xi \in K$, $s \in (0,\ve)$ and $|r|<\ve$ satisfying $|r-s| \leq 2|s|$,
the function
$$
\log |\lambda_{B}(z_{-}(r,s,t))|-\xi\log |z_{-}(r,s,t)|
$$
has critical points only at $t=\pm \pi/2$.
\end{enumerate}
\end{lem}

The proof will be given later.

\bigskip

\noindent{{\it End of the proof of the formula $\eqref{amp3AF}$}.}\hspace{5pt}
Assuming Lemma \ref{ellipseC}, let us prove the asymptotic formula $\eqref{amp3AF}$ for the case $\xi=\xi_{\pm}$.
First we consider the case where $\xi=\xi_{+}$.
Let $\{y_{n}\}$ be a sequence of integers satisfying $\xi_{n}:=y_{n}/n=\xi_{+}+O(1/n)$.
Let $K \subset (|a|,1)$ be a small closed interval centered at $\xi_{+}$ such that $\xi_{n} \in K$ for every sufficiently large $n$.
Take $\ve>0$ as in Lemma $\ref{ellipseC}$ and fix $s \in (-\ve,0)$.
Set $r_{n}:=r(\xi_{n})^{2}-r_{+}^{2}$. Then $r_{n} \to 0$ as $n \to \infty$.
Thus, for every sufficiently large $n$, the sequence $r_{n}$ satisfies $|r_{n}-s|<2|s|$.
For any $\eta \in K$, let $r_{\eta}:=r(\eta)^{2}-r_{+}^{2}$. Taking $K$ small enough if necessary,
we may assume that $|r_{\eta}-s| <2|s|$ and $|r_{\eta}|<\ve$ for every $\eta \in K$.
We set
\[
z(\eta,t):=z_{+}(r_{\eta},s,t)=\sqrt{r_{+}^{2}+s}\cos t +i r(\eta)\sin t,\quad -\pi \leq t \leq \pi.
\]
By Lemma $\ref{ellipseC}$, the function
\[
\phi_{\eta}(t)=\log |\lambda_{B}(z(\eta,t))|-\eta\log |z(\eta,t)|
\]
has critical points only at $t=\pm \pi/2$. Since
\[
\phi_{\eta}(\pi/2)=\log D(\eta)-\eta \log r(\eta)=-H_{Q}(\eta)/2 > \phi_{\eta}(-\pi/2)=-\log D(\eta)-\eta\log r(\eta),
\]
the function $\phi_{\eta}(t)$ attains its maximum at $t=\pi/2$ and its minimum at $t=-\pi/2$. As a contour $C$ in $\eqref{int41B}$, 
we take the ellipse
\[
C=C_{n}:z_{n}(t)=z(\xi_{n},t)=\sqrt{r_{+}^{2}+s}\cos t +ir(\xi_{n})\sin t\quad (-\pi \leq t \leq \pi).
\]
Then we can write the integral $I_{B1}(n)$ as
\[
I_{B1}(n)=\frac{e^{-nH_{Q}(\xi_{n})/2}}{2\pi i}\int_{-\pi}^{\pi}
\left(
\frac{z_{n}(t)}{r(\xi_{n})}
\right)^{-n\xi_{n}}
\left(
\frac{\lambda_{B}(z_{n}(t))}{D(\xi_{n})}
\right)^{n}f_{B\psi}(z_{n}(t))\frac{z_{n}'(t)}{z_{n}(t)}\,dt.
\]
By Lemma $\ref{ellipseC}$ and the discussion above, we have
\begin{equation}\label{phaseN}
\left|
\frac{z(\eta,t)}{r(\eta)}
\right|^{-\eta}
\left|
\frac{\lambda_{B}(z(\eta,t))}{D(\eta)}
\right| \leq 1,
\end{equation}
where the equality holds if and only if $t=\pi/2$. Thus one can choose a function $\chi(t) \in C_{0}^{\infty}(\mb{R})$
having a small support such that $\chi(t)=1$ near $t=\pi/2$ and
\[
\left|
\frac{z(\eta,t)}{r(\eta)}
\right|^{-\eta}
\left|
\frac{\lambda_{B}(z(\eta,t))}{D(\eta)}
\right| \leq e^{-c} \qquad (\eta \in K,~t\in {\rm supp}~\big(1-\chi(t)\big))
\]
with a constant $c>0$.
Since $\lambda_{B}(z(\eta,\pi/2))=iD(\eta)$ and $z(\eta,\pi/2)=ir(\eta)$,
taking $K$ and the support of $\chi(t)$ small enough,
we may assume that a branch of the logarithms $\log \lambda_{B}(z(\eta,t))$, $\log z(\eta,t)$ exist
when $t$ stays in a neighborhood of the support of $\chi(t)$ and $\eta \in K$.
Then the integral $I_{B1}(n)$ can be written as
\[
I_{B1}(n)=\frac{e^{-nH_{Q}(\xi_{n})/2}}{2\pi i}
\left(
\int e^{n\Phi(\xi_{n},t)}\chi(t)f_{B\psi}(z_{n}(t))\frac{z_{n}'(t)}{z_{n}(t)}\,dt+O(e^{-cn})
\right),
\]
where the smooth function $\Phi(\eta,t)$, which is defined on a neighborhood of $(\eta,t)=(\xi,\pi/2)$, is given by
\[
\Phi(\eta,t)=\log \lambda_{B}(z(\eta,t))-\eta \log z(\eta,t)-\log D(\eta)+\eta \log r(\eta).
\]
Note that $\Phi(\eta,\pi/2)=\pi i (1-\eta)/2$. By $\eqref{phaseN}$, we have $\re \Phi(\eta,t) \leq 0$, and
the equality holds if and only if $t=\pi/2$. 
The first and the second derivatives of $\Phi(\eta,t)$ are given by
\[
\begin{split}
\partial_{t}\Phi(\eta,t) & = \frac{z'}{z}\left(
z\frac{\lambda_{B}'(z)}{\lambda_{B}(z)}-\eta
\right),\\
\partial_{t}^{2}\Phi(\eta,t) & = -z \frac{\lambda_{B}'(z)}{\lambda_{B}(z)}+(z')^{2}
\frac{\lambda_{B}''(z)\lambda_{B}(z)-\lambda_{B}'(z)^{2}}{\lambda_{B}(z)^{2}}
+\eta\left(
1+\left(\frac{z'}{z}\right)^{2}
\right),
\end{split}
\]
where we set $z=z(\eta,t)$ and $z'=\partial_{t}z(\eta,t)$. By Lemma $\ref{cpB}$, $\partial_{t}\Phi(\eta,t)=0$ if and only if $t=\pi/2$.
By a direct computation, we find
\[
\partial_{t}^{2}\Phi(\eta,\pi/2)=-\frac{r_{+}^{2}+s}{r(\eta)^{2}}\frac{|a||b|^{2}\mu(r(\eta))}{(1+|a|^{2}\mu(r(\eta))^{2})^{3/2}}
=-\frac{r_{+}^{2}+s}{r(\eta)^{2}}\frac{(1-\eta^{2})\sqrt{\eta^{2}-|a|^{2}}}{|b|}<0.
\]
Note that each higher derivative of $\Phi(\eta,t)$ is bounded from above uniformly in $\eta \in K$ and $t$ in the support of $\chi(t)$,
and $|\partial_{t}^{2}\Phi(\eta,\pi/2)|$ is bounded from below when $\eta \in K$.
Since $\partial_{t}z(\eta,\pi/2)/z(\eta,\pi/2)=i\sqrt{r_{+}^{2}+s}/r(\eta)$, Theorem 7.7.5 in \cite{Ho} gives
\[
\begin{split}
I_{B1}(n) & =
\frac{e^{-nH_{Q}(\xi_{n})/2}}{2\pi i}\sqrt{
\frac{2\pi}{n}
\frac{r(\xi_{n})^{2}}{r_{+}^{2}+s}
\frac{|b|}{(1-\xi_{n}^{2})\sqrt{\xi_{n}^{2}-|a|^{2}}}
}e^{\pi i (n-y_{n})/2}
\left(
i\frac{\sqrt{r_{+}^{2}+s}}{r(\xi_{n})}f_{B\psi}(ir(\xi_{n}))+O(1/n)
\right)\\
& = \frac{e^{-nH_{Q}(\xi_{n})/2}}{\sqrt{2\pi n}}
\sqrt{
\frac{|b|}{(1-\xi_{+}^{2})\sqrt{\xi_{+}^{2}-|a|^{2}}}
}e^{\pi i (n-y_{n})/2}
\left(
f_{B\psi}(ir(\xi_{+}))+O(1/n)
\right).
\end{split}
\]
This proves the asymptotic formula $\eqref{amp3AF}$ with $F_{\psi}(\xi)=f_{B\psi}(ir(\xi_{+}))$ in the case $\xi=\xi_{+}$. 
When $\xi=\xi_{-}$, take $\ve>0$ and $s \in (0,\ve)$ as in Lemma $\ref{ellipseC}$, and take the ellipse
\[
z_{n}(t)=\sqrt{r_{-}^{2}+s}\cos t +i|r(\xi_{n})|\sin t
\]
for the contour $C$ in $\eqref{int41B}$.
Then, a similar argument as in the case $\xi=\xi_{+}$ will show the formula $\eqref{amp3AF}$ with $F_{\psi}(\xi)=f_{B\psi}(ir(\xi_{-}))$.
\hfill$\square$

\subsubsection{Proof of Lemma $\ref{ellipseC}$} 
Let $r, s, t$ be the parameters as in $\eqref{PR}$. For simplicity, we write $z(t)=z_{\pm}(r,s,t)$.
It is enough to show that the derivative of the function
\[
f(t)=f(r,s,t):=\log|\lambda_{B}(z(t))|^{2}-\xi\log |z(t)|^{2},\quad \xi \in K, 
\]
with respect to $t$ does not vanish when $-\pi \leq t \leq \pi$, $|t| \neq \pi/2$. By a simple computation, we have
\[
\partial_{t}f(t)=2\re
\left[
\frac{\partial_{t}z(t)}{z(t)}
\left(
z(t)\frac{\lambda_{B}'(z(t))}{\lambda_{B}(z)}-\xi
\right)
\right].
\]
We use the following notations to examine $\partial_{t}f(t)$.
\begin{equation}\label{auxC10}
\begin{gathered}
D=D(t):=|z(t)|=\sqrt{r_{\pm}^{2}+s \cos^{2}t +r\sin^{2}t},\quad
x(t)+iy(t):=|a|\phi(z(t)), \\
k=(r_{\pm}^{2}+s)^{1/2},\quad l=(r_{\pm}^{2}+r)^{1/2}, \quad \alpha(t)+i\beta(t)=\sqrt{1-|a|^{2}\phi(z(t))^{2}},\\
F(t)=(1-x(t)^{2}+y(t)^{2})^{2}+4x(t)^{2}y(t)^{2}.
\end{gathered}
\end{equation}
We have
\begin{equation}\label{auxC11}
\begin{gathered}
x=\frac{|a|k\phi(D)}{D}\cos t,\quad y=\frac{|a|l\mu(D)}{D}\sin t,\quad l^{2}-k^{2}=r-s,\\
2\alpha^{2}=(1-x^{2}+y^{2}) +\sqrt{F},\quad \alpha^{2}+\beta^{2}=\sqrt{F}.
\end{gathered}
\end{equation}
Then, 
\[
\begin{split}
z \frac{\lambda_{B}'}{\lambda_{B}} & = \frac{|a|}{D\sqrt{F}}\left[
(\alpha l \phi(D)\sin t -\beta k \mu(D) \cos t) -i(\alpha k \mu(D) \cos t +\beta l\phi(D) \sin t)
\right], \\
\frac{z'}{z} & = \frac{r-s}{D^{2}}\sin t \cos t +i \frac{kl}{D^{2}},\\
\re\left(
\frac{z'}{z}z\frac{\lambda_{B}'}{\lambda_{B}}
\right) & =
\frac{|a|(r-s)}{D^{3}\sqrt{F}}\sin t \cos t (\alpha l\phi(D)\sin t -\beta k \mu(D) \cos t) \\
& \hspace{30pt} +\frac{|a|kl}{D^{3}\sqrt{F}}(\alpha k \mu(D) \cos t +\beta l \phi(D) \sin t).
\end{split}
\]
From this we have
\begin{equation}\label{Fder1}
\frac{D^{3}\sqrt{F}}{2|a|}\partial_{t}f(t)
= R(t) (r-s) \sin t \cos t +kl S(t),
\end{equation}
where we set
\begin{equation}\label{Fder2}
R(t) =-\frac{\xi D \sqrt{F}}{|a|} +\alpha l\phi(D) \sin t -\beta k \mu(D) \cos t,\quad
S(t)=\alpha k\mu(D) \cos t +\beta l \phi(D) \sin t.
\end{equation}
We need to examine the function $S(t)$. Substituting
\[
\beta =-\frac{xy}{\alpha}=-\frac{|a|^{2}kl}{\alpha D^{2}}\phi(D)\mu(D) \sin t \cos t
\]
for $S(t)$, we have
\begin{equation}\label{funcA}
S(t)=\frac{k\mu(D)\cos t}{2\alpha}A(t),\quad
A(t)=
2\alpha^{2}-2\frac{|a|^{2}l^{2}}{D^{2}}\phi(D)^{2}\sin^{2}t.
\end{equation}
Substituting the concrete forms of $x(t)$, $y(t)$ for the second line of $\eqref{auxC11}$, we have
\begin{equation}\label{auxA1}
A(t)=\sqrt{F}+1-|a|^{2}\phi(D)^{2}-\frac{|a|^{2}l^{2}}{D^{2}}\sin^{2}t.
\end{equation}

\begin{lem}\label{estF}
There exists $\ve>0$ satisfying $\eqref{veRG1}$ such that,
for any $r,s,t$ in $\eqref{PR}$, one has
\[
\sqrt{F} \geq \frac{|a|^{2}l^{2}\sin^{2}t}{D^{2}}+\frac{|a|^{2}l^{2}\mu(D)^{2}\sin^{2}t}{D^{2}}+1-|a|^{2}\phi(D)^{2}.
\]
\end{lem}
\begin{proof}
Using the definition of $F$ in $\eqref{auxC10}$, we have
\[
F=\left(
1-|a|^{2}\phi(D)^{2}+\frac{|a|^{2}l^{2}}{D^{2}}\sin^{2}t
\right)^{2}+4\frac{|a|^{2}l^{2}\mu(D)^{2}}{D^{2}}\sin^{2}t,
\]
from which one has $\sqrt{F} \geq 1-|a|^{2}\phi(D)^{2}+|a|^{2}l^{2}\sin^{2}t/D^{2}$.
To improve this estimate, write
\[
F=\left(
1-|a|^{2}\phi(D)^{2}+\frac{|a|^{2}l^{2}}{D^{2}}\sin^{2}t+\frac{|a|^{2}l^{2}\mu(D)^{2}}{D^{2}}\sin^{2}t
\right)^{2}+G.
\]
Then, the function $G$ is computed as
\[
G=\frac{|a|^{2}l^{2}\mu(D)^{2}}{D^{2}}G_{1}(t) \sin^{2}t  +\frac{2|a|^{4}l^{2}\mu(D)^{2}}{D^{2}}G_{2}(t)\sin^{2}t ,
\]
where $G_{1}$, $G_{2}$ are given by
\[
G_{1}(t)=2-\frac{|a|^{2}l^{2}\mu(D)^{2}}{D^{2}}\sin^{2}t,\quad G_{2}(t)=\phi(D)^{2}-\frac{l^{2}}{D^{2}}\sin^{2}t.
\]
We would like to show that $\ve>0$ can be chosen so that $G \geq 0$ for any $r, s, t$ in $\eqref{PR}$.
For this sake, first note that, since $D^{2}=k^{2}\cos^{2}t+l^{2}\sin^{2}t$, we have
\begin{equation}\label{auxC12}
\frac{\sin^{2} t}{D^{2}} \leq \frac{1}{l^{2}}=\frac{1}{r_{\pm}^{2}+r} \leq \frac{1}{r_{\pm}^{2}-\ve},\quad -\pi \leq t \leq \pi.
\end{equation}
For $r, s, t$ satisfying $\eqref{PR}$, the following holds.
\[
\begin{cases}
1<(1+\ve)^{2}<r_{+}^{2}-\ve \leq D^{2} \leq r_{+}^{2}+\ve & \mbox{for $z_{+}(t)$}, \\
0<r_{-}^{2}-\ve \leq D^{2} \leq r_{-}^{2}+\ve <(1-\ve)^{2}<1 & \mbox{for $z_{-}(t)$}.
\end{cases}
\]
By this and $\eqref{auxC12}$, $G_{1}$ and $G_{2}$ are estimated as
\[
G_{1}(t) \geq 2-|a|^{2}\mu\left(\sqrt{r_{\pm}^{2} \pm \ve}\right)^{2},\quad G_{2}(t) \geq \phi(D)^{2}-1 \geq 0.
\]
Since $|a|\mu(r_{\pm})=\pm |b|$, we can take $\ve>0$ so that $2-|a|^{2}\mu\left(\sqrt{r_{\pm}^{2}\pm \ve}\right)^{2} >0$,
and then $G \geq 0$ as required.
\end{proof}

From $\eqref{auxA1}$ and Lemma $\ref{estF}$, we see that the function $A(t)$ has the estimate
\[
A(t) \geq A_{1}(t):=\frac{|a|^{2}l^{2}\mu(D)^{2}\sin^{2}t}{D^{2}}
+2(1-|a|^{2}\phi(D)^{2}).
\]

\begin{lem}\label{estA1}
One can choose $\ve>0$ satisfying $\eqref{veRG1}$ with the property that,
there exist positive constants $A$, $B$ such that, for any $r,s,t$ in $\eqref{PR}$,
we have
\[
A_{1}(t) \geq A|s|\cos^{2}t +B\sin^{2}t.
\]
\end{lem}

\begin{proof}
Set $f(t)=f(r,s,t):=s \cos^{2}t +r \sin^{2}t$ so that $D^{2}=r_{\pm}^{2}+f(t)$. Since $r, s, t$ satisfy $\eqref{PR}$, we have $|f(t)| \leq \ve <r_{\pm}^{2}$.
The Taylor expansion gives
\[
\frac{1}{D^{2}} =\frac{1}{r_{\pm}^{2}}-\frac{f(t)}{r_{\pm}^{4}}+O(f(t)^{2}),\quad
\frac{1}{D^{4}}=\frac{1}{r_{\pm}^{4}}-\frac{2f(t)}{r_{\pm}^{6}}+O(f(t)^{2}),
\]
where the constants in the terms $O(f(t)^{2})$ can be chosen uniformly in $r, s, t$.
Noting $|a|^{2}\phi(r_{\pm})^{2}=1$, we see
\[
A_{1}(t)=-\frac{|a|^{2}f(t)}{2}\left(
1-\frac{1}{r_{\pm}^{4}}
\right)+\frac{|a|^{2}l^{2}}{4}\sin^{2} t
\left(
1-\frac{1}{r_{\pm}^{2}}
\right)^{2} +\frac{|a|^{2}l^{2}f(t)}{2r_{\pm}^{4}}\sin^{2}t
\left(
1-\frac{1}{r_{\pm}^{2}}
\right)+O(f(t)^{2}).
\]
To rewrite this, note that $r_{-}=r_{+}^{-1}$ and hence
\[
1-\frac{1}{r_{\pm}^{2}}=\pm \frac{2|b|}{|a|r_{\pm}},\quad 1-\frac{1}{r_{\pm}^{4}}=\pm \frac{4|b|}{|a|^{2}r_{\pm}^{2}}.
\]
Substituting this for the above expression of $A_{1}(t)$, we can write
\begin{equation}\label{A1toG}
A_{1}(t)=\frac{|b|}{r_{\pm}^{2}}
\left(
g_{\pm}(t)+O(f(t)^{2})
\right),\quad
g_{\pm}(t)=\mp 2f(t)+|b|l^{2}\sin^{2}t \pm \frac{|a|l^{2}f(t)}{r_{\pm}^{3}}\sin^{2}t.
\end{equation}
We need to estimate $g_{\pm}(t)$ from below.
First we consider $g_{+}(t)$. Write
\[
g_{+}(t)=\left(
2-\frac{|a|l^{2}}{r_{+}^{3}}\sin^{2}t
\right)|s|\cos^{2}t +
\left(
|b|l^{2}-2r+\frac{|a|l^{2}r}{r_{+}^{3}}\sin^{2}t
\right)\sin^{2}t.
\]
Then we easily find
\[
2-\frac{|a|l^{2}}{r_{+}^{3}}\sin^{2}t \geq 2-\frac{|a|(r_{+}^{2}+\ve)}{r_{+}^{3}},\quad
|b|l^{2}-2r+\frac{|a|l^{2}r}{r_{+}^{3}}\sin^{2}t \geq |b|r_{+}^{2}-(2-|b|+|a|/r_{+})\ve.
\]
From this, one can choose $\ve>0$ small enough so that there exist positive constants $A$, $B$ with
\begin{equation}\label{estG}
g_{+}(t)=g_{+}(r,s,t) \geq A|s|\cos^{2}t +B\sin^{2}t
\end{equation}
for any $r, s, t$ satisfying the first line of $\eqref{PR}$.
The function $g_{-}(t)$ can also be estimated in a similar way. Indeed, $g_{-}(t)$ is written as
\[
g_{-}(t)=\left(
2-\frac{|a|l^{2}}{r_{-}^{3}}\sin^{2}t
\right)s \cos^{2}t +
\left(
|b|l^{2}+2r -\frac{|a|rl^{2}}{r_{-}^{3}}\sin^{2}t
\right)\sin^{2}t.
\]
In this turn, the parameters $r, s, t$ satisfy the second line of $\eqref{PR}$.
One has
\[
2-\frac{|a|l^{2}}{r_{-}^{3}}\sin^{2}t \geq 2-\frac{|a|}{r_{-}}-\frac{|a|}{r_{-}^{3}}\ve, \quad
|b|l^{2}+2r-\frac{|a|rl^{2}}{r_{-}^{3}}\sin^{2}t \geq |b|r_{-}^{2}-\left(2+|b|+\frac{|a|}{r_{-}}+\frac{|a|}{r_{-}^{3}}\right)\ve.
\]
Since $2-\frac{|a|}{r_{-}}>0$, this shows that there exist positive constants $A$, $B$ such that $\eqref{estG}$ holds
by replacing $g_{+}(t)$ with $g_{-}(t)$ for any $r, s, t$ satisfying the second line of $\eqref{PR}$.
Since $f(t)^{2}=O(\ve^{2})$, the assertion follows from the inequalities $\eqref{A1toG}$ and $\eqref{estG}$ by replacing $g_{+}(t)$ with $g_{\pm}(t)$.
\end{proof}

\noindent{\it End of the proof of Lemma $\ref{ellipseC}$.}\hspace{5pt} We set
\[
E(t)=2\alpha R(t) (r-s)\sin t +k^{2}l\mu(D) A(t),
\]
so that
\[
\frac{D^{3}\sqrt{F}}{2|a|}\partial_{t}f(t)=\frac{E(t)\cos t}{2\alpha}.
\]
Since the function $\alpha(t)R(t)$ is bounded for $r, s, t$ satisfying $\eqref{PR}$, we take $C>0$ such that
\[
|2\alpha(t)R(t)| \leq C.
\]
We divide the discussion into two cases.

\medskip

(1) Consider the case where $s$ satisfies the first line of $\eqref{PR}$.
Then we have $D>1$ and hence $\mu(D)>0$. Thus there exist positive constants $A$, $B$ such that
\[
k^{2}l\mu(D)A(t) \geq A|s|\cos^{2}t +B\sin^{2}t
\]
for any $r, s, t$ in $\eqref{PR}$. Hence, for $r,s$ with $|r-s| \leq 2|s|$, we have
\[
E(t) \geq (B-A|s|)\sin^{2}t -2C|s||\sin t| +A|s|.
\]
If $0<\ve <\min\{B/A,\,AB/(A^{2}+C^{2})\}$, then the polynomial
\[
p(T)=(B-A|s|)T^{2}-2C|s|T+A|s|
\]
is positive for every $T$, and hence $E(t)>0$ for every $r, s, t$ in the first line of $\eqref{PR}$.

\medskip

(2) Next, let us consider the case where $s$ satisfies the second line of $\eqref{PR}$.
In this case, $0<D<1$ and $\mu(D)<0$. Thus there exist positive constants $A$, $B$ such that
\[
k^{2}l \mu(D)A(t) \leq -A|s|\cos^{2}t -B\sin^{2}t
\]
for any $r, s, t$. Then, for $r,s$ with $|r-s| \leq 2|s|$, we see
\[
E(t) \leq -(B-A|s|)\sin^{2}t +2C|s||\sin t| -A|s|.
\]
If $0<\ve<\min\{B/A,\,AB/(A^{2}+C^{2})\}$, then the polynomial
\[
p(T)=-(B-A|s|)T^{2}+2C|s|T-A|s|
\]
is negative for every $T$, and hence $E(t)<0$ for every $r, s, t$ in the second line of $\eqref{PR}$.

\medskip

Therefore, for both cases, taking $\ve>0$ small enough, we see that
\[
f'(t) =\frac{|a|}{\alpha D^{3}\sqrt{F}}E(t) \cos t
\]
can vanish if and only if $\cos t=0$. This shows Lemma $\ref{ellipseC}$. \hfill$\square$
\subsection{Functions $F_{\psi}$ and $H_{Q}(\xi)$}
We have proved the formula $\eqref{amp3AF}$ with $F_{\psi}(\xi)$ given by $\eqref{FApsi1}$ and $\eqref{FBpsi1}$
for the case (A), (B) and (C), respectively.
Since
\[
\lambda_{A}(ir(\xi))=\lambda_{B}(ir(\xi))=iD(\xi)=i\frac{\sqrt{\xi^{2}-|a|^{2}}+|b|}{\sqrt{1-\xi^{2}}},\quad |a|<|\xi|<1,
\]
we have $u_{A}(ir(\xi))=u_{B}(ir(\xi))$, $v_{A}(ir(\xi))=v_{B}(ir(\xi))$,
and $f_{A\psi}(ir(\xi))=f_{B\psi}(ir(\xi))$. Therefore, we can write
\begin{equation}\label{FpsiF}
F_{\psi}(\xi)=\varphi(\xi)\ispa{u(\xi),\psi},
\end{equation}
where the function $\varphi(\xi)$ and the vector $u(\xi)$ are given by
\begin{equation}\label{FpsiF2}
\varphi(\xi)=\frac{i|a|(|a|r(\xi)^{-1}-D(\xi)^{-1})\varphi_{1}-ir(\xi)ab\varphi_{2}}{r(\xi)ab(D(\xi)+D(\xi)^{-1})},\quad
u(\xi)=
\begin{pmatrix}
\frac{ir(\xi)ab}{|a|} \\
iD(\xi)+i|a|r(\xi)^{-1}
\end{pmatrix}.
\end{equation}
Finally, let us prove the positivity and the convexity of the function $H_{Q}(\xi)$. 
\begin{lem}\label{rateQL}
The function $H_{Q}(\xi)$ defined by $\eqref{rateQ}$ is positive and convex on $|a|<|\xi|<1$. Moreover $H_{Q}(\xi)$ has no critical points.
\end{lem}
\begin{proof}
Since $H_{Q}(\xi)=H_{Q}(-\xi)$, it is enough to prove the lemma by assuming $|a|<\xi<1$.
For simplicity, we set $\delta(\xi):=H_{Q}(\xi)/2$ so that $\delta(\xi)=\xi \log r(\xi)-\log D(\xi)$, 
and hence $\delta'(\xi)=\log r(\xi)+\xi\frac{r'(\xi)}{r(\xi)}-\frac{D'(\xi)}{D(\xi)}$.
A direct computation shows
\[
\xi\frac{r'(\xi)}{r(\xi)}=\frac{D'(\xi)}{D(\xi)}=\frac{|b|\xi}{(1-\xi^{2})\sqrt{\xi^{2}-|a|^{2}}}.
\]
Since $1<r(\xi)$ for $\xi>|a|$, we see $\delta'(\xi)=\log r(\xi)>0$. It is easy to show
\[
\lim_{\xi \to |a|+0}\delta(\xi)=0,\quad \lim_{\xi \to 1-0}\delta(\xi)=-\log |a|>0,
\]
and which implies $\delta(\xi)=H_{Q}(\xi)/2>0$. We also have
\[
\delta''(\xi)=\frac{r'(\xi)}{r(\xi)}=\frac{|b|}{(1-\xi^{2})\sqrt{\xi^{2}-|a|^{2}}}>0,
\]
which shows the convexity of $H_{Q}(\xi)$.
\end{proof}

\subsection{Proof of Theorem ${\bf \ref{hidden}}$} Taking the modulus square of $\eqref{amp3AF}$ with $\psi=e_{1}$ and $\psi=e_{2}$,
and summing up these formulae, we have the asymptotic formula $\eqref{hiddenAF}$ with
\begin{equation}\label{funcG}
G(\xi)=|F_{e_{1}}(\xi)|^{2}+|F_{e_{2}}(\xi)|^{2}=|\varphi(\xi)|^{2}\left(
|\ispa{u(\xi),e_{1}}|^{2}+|\ispa{u(\xi),e_{2}}|^{2}
\right)=|\varphi(\xi)|^{2}\|u(\xi)\|^{2},
\end{equation}
which is a non-negative function, where the function $\varphi(\xi)$ and the vector $u(\xi)$ are given in $\eqref{FpsiF2}$. 
This completes the proof of Theorem $\ref{hidden}$.

\end{document}